\documentclass[english]{amsart}
\usepackage{babel}
\usepackage[margin=1.2in]{geometry}
\usepackage{amsmath,amssymb,amsfonts,latexsym,cancel}
\usepackage{rawfonts}
\usepackage{pictexwd}

\usepackage{caption}

\usepackage{epsfig}
\usepackage{pst-grad} 
\usepackage{pst-plot} 

\newcommand{\R}{\mathbb{R}}

\newcommand{\eps}{\varepsilon}

\theoremstyle{plain}
\newtheorem{defi}{Definition}[section]
\newtheorem{proposition}[defi]{Proposition}
\newtheorem{theorem}[defi]{Theorem}
\newtheorem{corollary}[defi]{Corollary}
\newtheorem{lemma}[defi]{Lemma}

\newtheorem{remark}[defi]{Remark}

\theoremstyle{definition}

\theoremstyle{remark}

\numberwithin{equation}{section}
\numberwithin{figure}{section}
\numberwithin{table}{section}

\begin{document}

\title[Nontrivial Touchdown Sets]{No touchdown at points of small permittivity
 and nontrivial touchdown sets  for the MEMS problem}

\author{Carlos Esteve}
\address{Universit\'{e} Paris 13, Sorbonne Paris Cit\'{e}, CNRS UMR 7539,
Laboratoire Analyse, G\'{e}om\'{e}trie et Applications, 93430 Villetaneuse, France.
\newline {\tt E-mail address: esteve@math.univ-paris13.fr}
}

\author{Philippe Souplet}
\address{Universit\'{e} Paris 13, Sorbonne Paris Cit\'{e},  CNRS UMR 7539,
Laboratoire Analyse, G\'{e}om\'{e}trie et Applications, 93430 Villetaneuse, France.
\newline {\tt E-mail address: souplet@math.univ-paris13.fr}
}

\date{\today}

\begin{abstract} 
We consider a well-known model for micro-electromechanical systems (MEMS)
with variable dielectric permittivity, involving a parabolic equation with singular nonlinearity.
We study the touchdown, or quenching, phenomenon.
Recently, the question whether or not touchdown can occur at zero points of the permittivity profile $f$,
which had long remained open, was answered negatively
for the case of interior points.

The first aim of this article is to go further by considering the same question at points of positive but small permittivity.
We show that, in any bounded domain, touchdown cannot occur
at an interior point where the permittivity profile is suitably small.
We also obtain a similar result in the boundary case, 
under a smallness assumption on 
 $f$ in a neighborhood of the boundary.
This allows in particular to construct $f$ producing touchdown sets concentrated near any given sphere.

Our next aim is to obtain more information on the structure and properties of the touchdown set.
In particular, we show that the touchdown set need not in general be localized 
near the maximum points of the permittivity profile $f$.
In the radial case in a ball, we actually show the existence of ``M''-shaped profiles $f$ for which the touchdown set is located 
{\it far away} from the maximum points of $f$
and we even obtain strictly convex $f$ for which touchdown occurs only at the unique {\it minimum} point of $f$.
These results give analytical confirmation of some numerical simulations from the book
{\it [P.~Esposito, N. Ghoussoub, Y. Guo,
Mathematical analysis of partial differential equations modeling electrostatic MEMS,
Courant Lecture Notes in Mathematics, 2010]}
and solve some of the open questions therein.
They also show that some kind of smallness condition as above
cannot be avoided in order to rule out touchdown at a point.

On the other hand, we construct profiles $f$ producing more complex behaviors:
 in any bounded domain the touchdown set may be concentrated near two arbitrarily given points,
or two arbitrarily given $(n-1)$-dimensional spheres in a ball.
These examples are obtained as a consequence of stability results
for the touchdown time and touchdown set under small perturbations of the permittivity profile.

\end{abstract}

\maketitle

\section{Introduction}

\subsection{Mathematical problem and physical background}

We consider the problem
\begin{equation}\label{quenching problem}
\left\lbrace \begin{array}{rrl}
u_t - \Delta u = f(x)(1-u)^{-p}, & t>0, & x\in \Omega, \\
u = 0, & t>0, &  x\in \partial \Omega,\\
u(0,x) = 0, & x \in \Omega, &
\end{array} \right.
\end{equation}
where $\Omega$ is a smooth bounded domain in $\mathbb{R}^n$, $n\geq 1$, $p>0$ and $f\in E$, where
\begin{equation}\label{hypf}
E=\bigl\{f:\overline\Omega\to[0,\infty);\ f \hbox{ is H\"older continuous}\bigr\}.
\end{equation}
Problem (\ref{quenching problem}) with $p=2$ is a known model for micro-electromechanical devices (MEMS)
and has received a lot attention in the past 15 years.
An idealized version of such device consists of two conducting plates, connected to an electric circuit.
The lower plate is rigid and fixed while the upper one is elastic and fixed only at the boundary.
Initially the plates are parallel and at unit distance from each other.
When a voltage (difference of potential between the two plates) is applied, the upper
plate starts to bend down and, if the voltage is large enough,
the upper plate eventually touches the lower one. This is called {\it touchdown} 
phenomenon.
Such device can be used for instance as an actuator, a microvalve (the touching-down part closes the valve),
or a fuse.

In the mathematical model, $u=u(t,x)$ measures the vertical deflection of the upper plate
and the function $f(x)$ represents the dielectric permittivity of the material
which, as a key feature, may be possibly inhomogeneous.
(Actually $f$ is also proportional to the -- constant -- applied voltage.)

It is well known that problem (\ref{quenching problem}) admits a unique maximal classical solution~$u$.
We denote its maximal existence time by $T=T_f\in (0,\infty]$. Moreover,
under some largeness assumption on $f$, it is known that the maximum of $u$ reaches the value $1$
at a finite time, so that $u$ ceases to exist in the classical sense, i.e. $T<\infty$.
This property, known as quenching, is the mathematical counterpart of the touchdown phenomenon. 

A point $x = x_0 \in\overline\Omega$
is called a {\it touchdown} or {\it quenching point} if there exists 
a sequence $\{(t_n, x_n)\} \in (0, T)\times \Omega$ such that
$$x_n\to x_0,\ \ t_n\uparrow T\ \ \hbox{and}\ \ u(x_n,t_n)\to 1 \ \hbox{as}\ n\to\infty.$$
The set of all such points is closed. It is
called the {\it touchdown} or {\it quenching set}, denoted by $\mathcal{T}=\mathcal{T}_f
 \subset \overline\Omega$. 

MEMS problems, including system (\ref{quenching problem}) and the related touchdown issues,
 have received considerable attention in the physical and engineering as well as 
in the mathematical communities. We refer to \cite{EGG}, \cite{JAP-DHB02} for more details on the physical background, 
and to, e.g., \cite{G97}, \cite{FT00} \cite{PT01}, \cite{YG-ZP-MJW05}, \cite{FMPS07}, \cite{GG07}, \cite{GG08}, \cite{GHW08}, \cite{G08}, 
\cite{G08A},  \cite{KMS08},  \cite{GK}, \cite{YZ10}, \cite{G14}, \cite{GS} for mathematical studies.
See also \cite{Phi87}, \cite{L89}, \cite{DL89}, \cite{G1}, \cite{FLV}, \cite{FG93} for earlier mathematical work on the case of constant $f$.

As a question of particular interest,
it has long remained open whether touchdown could occur at zero points of the permittivity profile.
This has been answered negatively in \cite{GS} for the case of interior points.
This is by no means obvious since, for the analogous blowup problem $u_t - \Delta u = f(x)u^p$
with $f(x)=|x|^\sigma$, examples of solutions blowing up at the origin have been constructed in \cite{FT00}, \cite{GS11}
for suitable $\sigma>0$, $p>1$ and suitable initial data $u_0\ge 0$.

To go further, natural questions are then:

$\bullet$ can one rule out touchdown at points of positive but {\it small} permittivity ?

$\bullet$ can one obtain more information on the structure and properties of the touchdown set ?

These are the main motivations of the present article.

\subsection{Results (I): no touchdown at points of small permittivity}

Our first main result shows that touchdown cannot occur at an interior point of small permittivity $f(x_0)$,
and we provide a suitable smallness condition in terms of $f$ and $x_0$.
In the sequel, we denote by
$$\delta (x) := \text{dist} (x,\partial\Omega), \qquad x\in\overline\Omega,$$
the function distance to the boundary.

\begin{theorem}[No touchdown at interior points of small permittivity]\label{local result dim n}
Let $p>0$, $\Omega\subset \mathbb{R}^n$ a smooth bounded domain and $f\in E$.
Assume
\begin{equation}\label{condMB}
\left\lbrace \begin{array}{rrl}
&&T_f\le M,\quad \|f\|_\infty\le M,\quad f\ge r\chi_B, \\
\noalign{\vskip 1mm}
&&\hbox{where $M, r>0$ and $B\subset\Omega$ is a ball of radius $r$.}
\end{array} \right.
\end{equation}
There exists $\gamma_0>0$ depending only on $p, \Omega, M, r$ such that, 
for any $x_0\in \Omega$, if
\begin{equation}\label{hyp local result dim n}
f(x_0)<\gamma_0{\hskip 1pt}\delta^{p+1}(x_0),
\end{equation}
then $x_0\not\in \mathcal{T}_f$.
\end{theorem}

As a drawback of Theorem~\ref{local result dim n}, boundary points are not covered,
and the threshold value vanishes when $x_0$ approaches the boundary.
Actually, it remains an open problem whether touchdown can occur on the boundary,
including at boundary points of zero permittivity.
Some partial results can be found in \cite{G08}, \cite{GS},
where $f(x)$ is assumed to either be monotonically decreasing
or to vanish sufficiently fast, as $x$ approaches the boundary.
Our second main result gives another contribution to that question.
It shows that touchdown can be localized in any compact subdomain of $\Omega$
under the assumption that $f$ is small enough outside this subdomain.
It is thus of a more global nature than the local criterion in Theorem~\ref{local result dim n} for interior points.
As a consequence, it rules out touchdown on the boundary
when $f$ is small enough on a neighborhood of the boundary.
We stress that for this result, unlike in  \cite{G08}, \cite{GS}, we do not require any monotonicity or decay of $f$ near $\partial\Omega$.

\begin{theorem}[No touchdown for small permittivity near the boundary]\label{global result dim n}
 Let $p>0$, $\Omega\subset \mathbb{R}^n$ a smooth bounded domain and $f\in E$.
Assume \eqref{condMB}. There exists $\gamma_0 >0$ depending only on $p,\Omega,M,r$ such that, for any $\omega \subset\subset \Omega$, if
\begin{equation}\label{hyp global result dim n}
\sup_{x\in\overline\Omega\setminus \omega}f(x)< \gamma_0 {\hskip 1pt} {\rm dist}^{p+1}(\omega,\partial\Omega),
\end{equation}
then $\mathcal{T}_f\subset \omega$.
\end{theorem}

In view of Theorems~\ref{local result dim n} and \ref{global result dim n},
it is a natural question whether smallness conditions, such as \eqref{hyp local result dim n} and \eqref{hyp global result dim n},
 are actually necessary, 
or whether touchdown could be shown to occur only at or near the maximum points of the permittivity profile $f$.
This, among other related issues, is the subject of our next subsection,
where a number of results on the structure and properties of the touchdown set are obtained.


\subsection{Results (II): Nontrivial touchdown sets and ``M''-shaped profiles.}

We will pay special attention to the following class of permittivity profiles.
For $\Omega=B_R\subset\mathbb{R}^n$ ($n\geq 1$), we call ``M''-shaped permittivity profile a function $f$ such that
\begin{equation}\label{defMshaped}
\begin{array}{ll}
&\hbox{$f$ is radially symmetric, nondecreasing in $|x|$ on $[0,L]$} \\
\noalign{\vskip 1mm}
&\hbox{and nonincreasing in $|x|$ on $[L,R]$, for some $L\in (0,R)$.}
\end{array}
\end{equation}

In the book \cite[Section 7.4]{EGG}, for particular ``M''-shaped profiles, numerical simulations were carried out, which suggest some interesting phenomena regarding  the location of touchdown points. In this paper we are able to confirm some of them by rigorous analytical arguments.
In this connection, we shall construct ``M''-shaped profiles, and variants thereof, giving rise to various types of touchdown sets: single-point, touchdown set concentrated near a sphere, near two points, near two spheres.
We point out that such properties may be useful in the practical design of MEMS devices,
at least on a qualitative level.

To begin with, as a consequence of Theorems~\ref{local result dim n} and \ref{global result dim n}, we have the following corollary.

\begin{corollary}\label{cor 1}
Let $p>0$, $\Omega=B_R\subset\R^n$.

(i) {\rm (Touchdown containing a sphere.)}
Let $f\in E$ be an ``M''-shaped profile, i.e.~(\ref{defMshaped}) holds, and assume \eqref{condMB}.
If $f(0)$ is small enough (depending only on $p,n,R,M,r$), then $0$ is not a touchdown point.
In particular, $\mathcal{T}_f$ contains an $(n-1)$-dimensional sphere. 

(ii) {\rm (Touchdown concentrated near a given sphere.)}
Let $r>0$ and $0<\eps<\min(r,R-r)$.
There exist two-bump, ``M''-shaped profiles $f$ such that $T_f<\infty$ and 
$$\mathcal{T}_f\subset \{ r-\eps< |x| < r+\eps\}.$$
More precisely, there exist $\eta, A>0$, depending only on $p,R,r,\eps$, such that this is true 
for any radially symmetric $f\in E$ satisfying
$$
\left\{\begin{array}{llll}
f(x)\ge A, &\text{for} \ |x|\in[r-\eps/2,r+\eps/2], \\
\noalign{\vskip 2mm}
f(x)\le\eta, & \text{for} \ |x|\in[0,r-\eps]\cup [r+\eps,R].
\end{array}
\right.
$$
\end{corollary}

\eject

\begin{figure}[h]
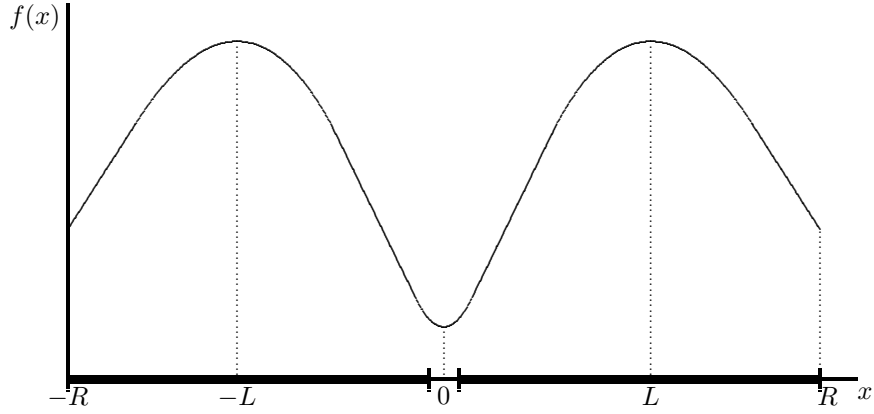

\[
\beginpicture
\setcoordinatesystem units <1cm,1cm>
\setplotarea x from -6 to 6, y from -1 to 5

\setdots <0pt>
\linethickness=1pt
\putrule from -5 0 to 5.5 0
\putrule from -5 0 to -5 5

\setlinear
\plot -5 2  -4.1 3.4 /
\plot -1.5 3.4  -0.38 1.1 / 
\plot 5 2  4.1 3.4 /
\plot 1.5 3.4  0.38 1.1 / 

\setquadratic
\plot -4.1 3.4   -2.75 4.5 -1.5 3.4 /
\plot 4.1 3.4   2.75 4.5 1.5 3.4 /
\plot -0.38 1.1  0 0.7  0.38 1.1 / 

\setdots <2pt>
\setlinear
\plot -2.75 0  -2.75 4.5 /
\plot 0 0  0 0.7 /
\plot 2.75 0  2.75 4.5 /
\plot 5 0  5 2 /

\put {$x$} [lt] at 5.5 -.1
\put {$R$} [ct] at 5.1 -.1
\put {$L$}   [ct] at 2.75 -.1
\put {$0$}   [ct] at 0 -.1
\put {$-L$}  [ct] at -2.75 -.1
\put {$-R$}  [ct] at -5 -.1
\put {$f(x)$} [rt] at -5.1 5

\setdots <0pt>
\linethickness=3pt
\putrule from -5 0 to -0.2 0
\putrule from 0.2 0 to 5 0 

\linethickness=1pt
\putrule from -5 .15 to -5 -.15 
\putrule from 0.2 .15 to 0.2 -.15 
\putrule from 5 .15 to  5 -.15 
\putrule from -0.2 .15 to  -0.2 -.15 

\endpicture
\]
\caption{Illustration of Corollary \ref{cor 1}(i) --  in this and the subsequent figures, the touchdown set 
must be a subset of the fat lines}
\label{fig cor 1}
\end{figure}

\begin{figure}[h]
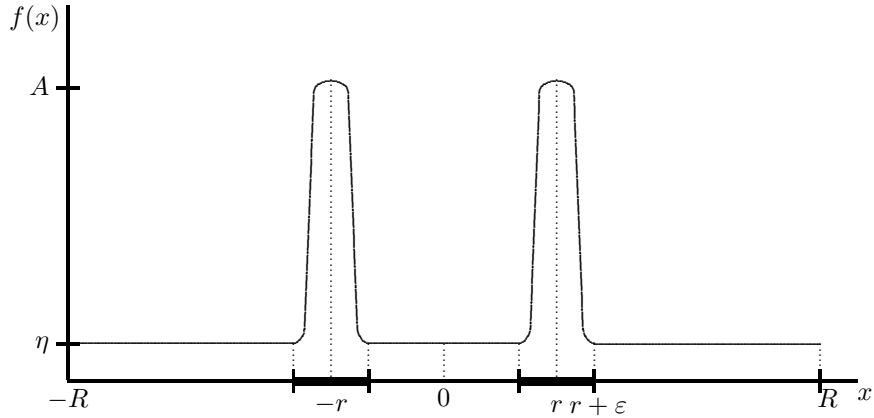

\[
\beginpicture
\setcoordinatesystem units <1cm,1cm>
\setplotarea x from -6 to 6, y from -1 to 5

\setdots <0pt>
\linethickness=1pt
\putrule from -5 0 to 5.5 0
\putrule from -5 0 to -5 5

\setlinear
\plot -1.27 3.8  -1.15 0.7 / 
\plot -1 0.5  1 0.5 / 
\plot -1.85 0.7  -1.73 3.8 / 
\plot -5 0.5  -2 0.5 / 
\plot 1.27 3.8  1.15 0.7 / 
\plot 1 0.5  1 0.5 / 
\plot 1.85 0.7  1.73 3.8 / 
\plot 5 0.5  2 0.5 / 

\setquadratic
\plot -1.5 4  -1.32 3.94  -1.27 3.8 /
\plot -1.15 0.7  -1.1 0.57  -1 0.5 / 
\plot -2 0.5  -1.9 0.57  -1.85 0.7 / 
\plot -1.5 4  -1.68 3.94  -1.73 3.8 /
\plot 1.5 4  1.32 3.94  1.27 3.8 /
\plot 1.15 0.7  1.1 0.57  1 0.5 / 
\plot 2 0.5  1.9 0.57  1.85 0.7 / 
\plot 1.5 4  1.68 3.94  1.73 3.8 /

\setdots <2pt>
\setlinear
\plot -1.5 0  -1.5 4 /
\plot -1 0  -1 0.5 /
\plot -2 0  -2 0.5 /
\plot 0 0  0 0.5 /
\plot 1.5 0  1.5 4 /
\plot 1 0  1 0.5 /
\plot 2 0  2 0.5 /
\plot 5 0  5 0.5 /

\put {$x$} [lt] at 5.5 -.1
\put {$R$} [ct] at 5.1 -.1
\put {$0$}   [ct] at 0 -.1
\put {$-r$}  [ct] at -1.5 -.19
\put {$r$}  [ct] at 1.5 -.26
\put {$r+\eps$}  [ct] at 2.05 -.21
\put {$-R$}  [ct] at -5 -.1
\put {$f(x)$} [rt] at -5.1 5

\put {$\eta$} [rt] at -5.25 0.6
\put {$A$} [rt] at -5.25 4.05

\setdots <0pt>
\linethickness=3pt
\putrule from -2 0 to -1 0
\putrule from 2 0 to 1 0

\linethickness=1pt
\putrule from -2 .15 to -2 -.15 
\putrule from -1 .15 to  -1 -.15 
\putrule from 2 .15 to 2 -.15 
\putrule from 1 .15 to  1 -.15 
\putrule from 5 .15 to  5 -.15 
\putrule from -5.15 0.5 to -4.85 0.5 
\putrule from -5.15 3.9 to -4.85 3.9

\endpicture
\]
\caption{Illustration of Corollary \ref{cor 1}(ii)}
\label{fig cor 1b} 
\end{figure}

When $\Omega$ is a ball and $f$ is constant or radial nonincreasing, it is well known that touchdown can occur
only at the origin (see \cite{DL89}, \cite{G1}, \cite{G08}).
In particular, this is the case if we take $f(0)=f(L)$ instead of $f(0)$ small in Corollary~\ref{cor 1}(i).
A natural question is then, whether the assumption ``$f(0)$ small enough'' in Corollary~\ref{cor 1}(i)
 could be replaced by $f(0)<f(L)$. 
The following theorem, which shows the stability of single point touchdown under suitable perturbation of $f$,
answers this question negatively.
In the sequel we denote $\|\cdot\|_q=\|\cdot\|_{L^q(\Omega)}$ and
\begin{equation}\label{DefMu0}
\mu_0(p,n):= \dfrac{p^p}{(p+1)^{p+1}}\lambda_1,
\end{equation}
where $\lambda_1$ is the first eigenvalue of $-\Delta$ in $H_0^1(B_1)$ and $B_1$ is the unit ball in $\R^n$.

\goodbreak
\eject

\begin{theorem}[Stability of single point touchdown under perturbation]\label{single point quenching}
Let $p>0$, $\Omega=B_R\subset\R^n$, $1\le q \le \infty$ with $q>\frac{n}{2}$, $M>0$, $\rho\in (0,R)$.
Let $f\in E \cap\, C^1(\overline B_\rho)$ be radially symmetric nonincreasing, with $f(r)>\mu_0(p,n)\rho^{-2}$ on $\overline{B}_\rho$.
There exists $\eps>0$ such that, if $g\in E\cap C^1(\overline B_\rho)$ is radially symmetric and satisfies 
\begin{eqnarray}
&&\|g\|_\infty\le M, \label{gmaxM} \\
&& -M\le g'(r) \leq \eps r, \quad \text{for all $r\in[0,\rho]$}, \label{gprimerho} \\
&&\hbox{$\|g-f\|_q \leq \varepsilon$}, \label{gLq}
\end{eqnarray}
then $T_g<\infty$ and $\mathcal{T}_g=\{0\}$.
\end{theorem} 

\goodbreak

As a direct consequence of Theorem \ref{single point quenching},
there exist genuine ``M''-shaped profiles $g$ (i.e., such that $g(0)<g(L)$, with $g(0)$ close to $g(L)$), for which touchdown occurs at the single point $x=0$
(see figure~\ref{fig th 1}).
This shows that some kind of smallness condition, such as \eqref{hyp local result dim n} or \eqref{hyp global result dim n}, is required in order to rule out touchdown in a given region of the domain.

In turn, this provides examples of profiles $f$ for which the touchdown 
set is located {\it far away} from the maximum points of $f$.
It also shows that the radial nonincreasing monotonicity of $f$ is sufficient but {\it not necessary}
for single point touchdown at the origin.
This confirms some of the numerical predictions from \cite{EGG} (see \cite[Remark 7.4.2]{EGG}).
Such a behavior must be interpreted as an effect of the diffusion (and of the boundary conditions),
since in the absence of diffusion the explicit computation immediately shows that touchdown 
occurs only at the maximum points of $f$.

\begin{figure}[h]
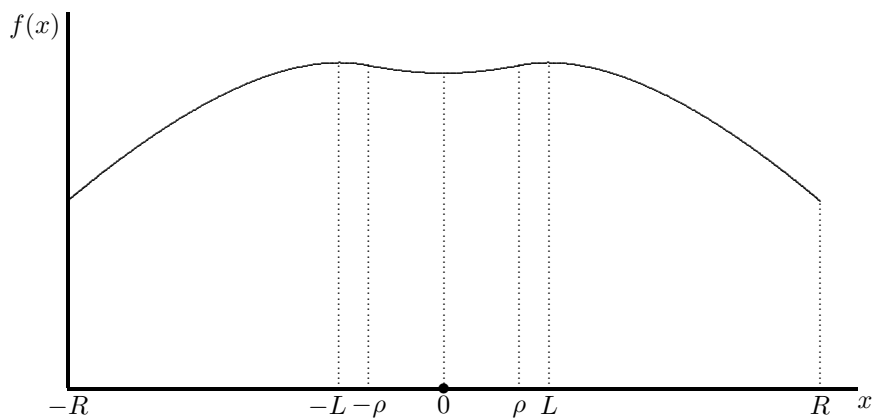

\[
\beginpicture
\setcoordinatesystem units <1cm,1cm>
\setplotarea x from -6 to 6, y from -1 to 5

\setdots <0pt>
\linethickness=1pt
\putrule from -5 0 to 5.5 0
\putrule from -5 0 to -5 5

\setquadratic

\plot -5 2.5  -2.75 4  -1 4.3 /
\plot 5 2.5  2.75 4  1 4.3 /
\plot -1 4.3  0 4.2  1 4.3 /

\setdots <2pt>
\setlinear
\plot 1 0  1 4.3 /
\plot -1 0  -1 4.3 /
\plot 0 0  0 4.2 /
\plot -1.4 0  -1.4 4.3 /
\plot 1.4 0  1.4 4.3 /
\plot 5 0  5 2.5 /

\put {$x$} [lt] at 5.5 -.1
\put {$R$} [ct] at 5 -.1
\put {$\rho$}   [ct] at 1 -.15
\put {$L$}   [ct] at 1.4 -.1
\put {$0$}   [ct] at 0 -.1
\put {$-\rho$}  [ct] at -1 -.1
\put {$-L$}   [ct] at -1.55 -.1
\put {$-R$}  [ct] at -5 -.1
\put {$f(x)$} [rt] at -5.1 5

\put {$\bullet$} [cc] at 0 0

\endpicture
\]
\caption{An illustration of Theorem \ref{single point quenching} in one space dimension for an ``M''-shaped profile}
\label{fig th 1}
\end{figure}

Another, rather surprising, consequence of Theorem \ref{single point quenching}, is the possibility of constructing 
{\it strictly convex} profiles
producing single point touchdown at the unique {\it minimum} point of $f$. 
Indeed, let $f_\lambda(x)$ be the function defined in $\Omega=B(0,R)$ by
$$
f_\lambda(x) := \mu + \lambda \dfrac{|x|^2}{R^2}, \qquad \text{with } \mu > \mu_0(p,n)\rho^{-2} \text{ and } \lambda \ge 0. 
$$
We see that $f_0$ is radially nonincreasing and, for $\lambda>0$ small enough, $f_\lambda$ satisfies the hypothesis of the Theorem.
Therefore, the only touchdown point is the origin, i.e. the unique minimum point of $f_\lambda$ (see figure \ref{fig th 2}).
This solves negatively the open question in \cite[Section 7.5]{EGG},
on whether the touchdown set must consist of an $(n-1)$ dimensional sphere when $f(x) = f(|x|)$ is increasing in $|x|$.
This example also shows that the monotonicity or decay hypotheses on $f$ near the boundary are not necessary in general
for the compactness of $\mathcal{T}_f$.

\bigskip

\begin{figure}[h]
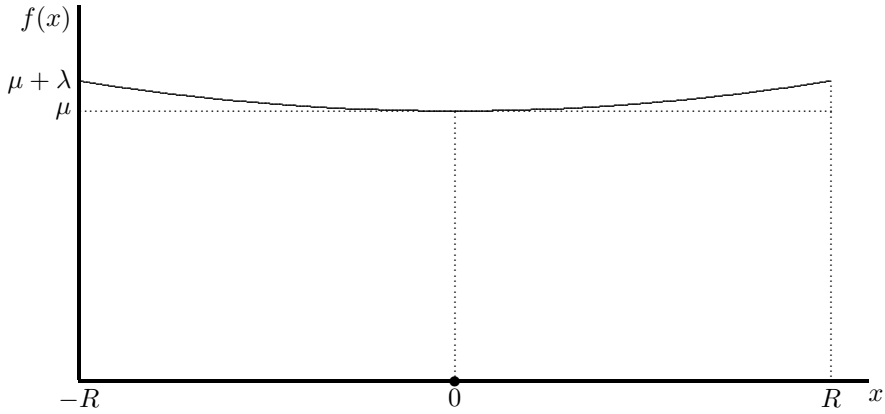

\[
\beginpicture
\setcoordinatesystem units <1cm,1cm>
\setplotarea x from -6 to 6, y from -1 to 5

\setdots <0pt>
\linethickness=1pt
\putrule from -5 0 to 5.5 0
\putrule from -5 0 to -5 5

\setquadratic

\plot -5 4  0 3.6  5 4 /

\setdots <2pt>
\setlinear
\plot 0 0  0 3.6 /
\plot 5 0  5 4 /

\plot 5 3.6  -5  3.6 /

\put {$x$} [lt] at 5.5 -.1
\put {$R$} [ct] at 5 -.1
\put {$0$}   [ct] at 0 -.1
\put {$-R$}  [ct] at -5 -.1
\put {$f(x)$} [rt] at -5.1 5
\put {$\mu $} [rc] at -5.1 3.6
\put {$\mu + \lambda $} [rc] at -5.1 4

\put {$\bullet$} [cc] at 0 0

\endpicture
\]
\caption{An illustration of Theorem \ref{single point quenching} for a strictly convex profile}
\label{fig th 2}
\end{figure}

\bigskip

In Corollary \ref{cor 1}(ii) we saw that the touchdown set can be concentrated 
near any $(n-1)$-dimensional sphere, where $f$ achieves its maxima.
As a consequence of the following result, which shows the stability of unfocused touchdown concentrated near the origin,
we obtain profiles $g$ whose touchdown set
contains an $(n-1)$-dimensional sphere
and is arbitrarily concentrated near the origin, {\it far away} from the maxima of $g$.
Such $g$ can take the form of an ``M''-shaped profile with 
a narrow ``well'' near the origin (see fig.~\ref{fig th 2 b}). 

\begin{theorem}[Stability of unfocused touchdown concentrated near the origin]\label{theorem one well}
Let $p>0$, $\Omega=B_R\subset\R^n$, $1\le q \le \infty$ with $q>\frac{n}{2}$, $0<\eta<R$.
Let $M, \rho>0$, $B:=B(x_0,\rho)\subset\Omega$ and $\mu>\mu_0(n,p)\rho^{-2}$.
Let $f\in E$ be radially symmetric nonincreasing.
There exists $\eps>0$ such that, if $g\in E$ is radially symmetric and satisfies 
\begin{eqnarray*}
&&\mu \chi_B\le g\le M, \\
\noalign{\vskip 1mm}
&&\hbox{$g(0)<\eps$,} \\
\noalign{\vskip 1mm}
&&\hbox{$\|g-f\|_q \leq \varepsilon$},
\end{eqnarray*}
then $T_g<\infty$ and $\mathcal{T}_g\subset B_\eta\setminus\{0\}$.
In particular $\mathcal{T}_g$ contains at least an $(n-1)$-dimensional sphere.
\end{theorem} 

\bigskip

\begin{figure}[h]
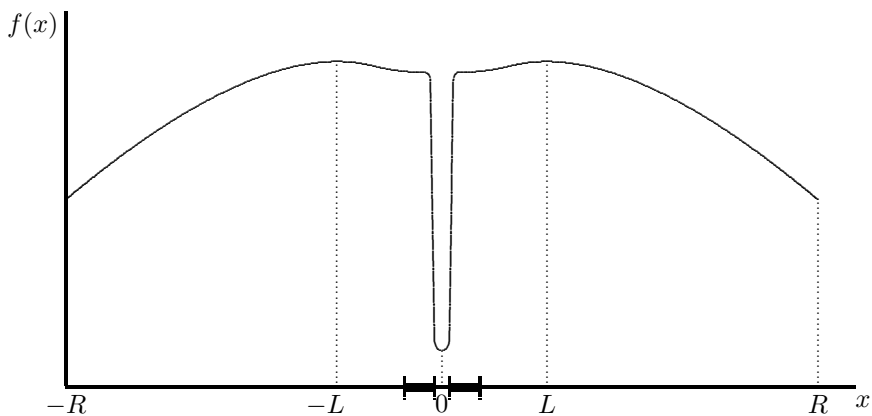

\[
\beginpicture
\setcoordinatesystem units <1cm,1cm>
\setplotarea x from -6 to 6, y from -1 to 5

\setdots <0pt>
\linethickness=1pt
\putrule from -5 0 to 5.5 0
\putrule from -5 0 to -5 5

\setlinear

\plot -0.15 4  -0.1 0.7 /
\plot 0.15 4  0.1 0.7 /

\setquadratic
\plot -0.25 4.2  -0.17 4.15  -0.15 4 /
\plot -0.1 0.7  -0.08 0.55  0 0.5 / 
\plot 0.25 4.2  0.17 4.15  0.15 4 /
\plot 0.1 0.7  0.08 0.55  0 0.5 / 

\plot 5 2.5  2.75 4  1 4.3 /
\plot 1 4.3  0.6 4.22  0.25 4.2 /
\plot -5 2.5  -2.75 4  -1 4.3 /
\plot -1 4.3  -0.6 4.22  -0.25 4.2 /

\setdots <2pt>
\setlinear
\plot 0 0  0 0.5 /
\plot -1.4 0  -1.4 4.3 /
\plot 1.4 0  1.4 4.3 /
\plot 5 0  5 2.5 /

\put {$x$} [lt] at 5.5 -.1
\put {$R$} [ct] at 5 -.1
\put {$L$}   [ct] at 1.4 -.1
\put {$0$}   [ct] at 0 -.1
\put {$-L$}   [ct] at -1.55 -.1
\put {$-R$}  [ct] at -5 -.1
\put {$f(x)$} [rt] at -5.1 5

\setdots <0pt>
\linethickness=3pt
\putrule from -0.5 0 to -0.1 0
\putrule from 0.1 0 to 0.5 0 

\linethickness=1pt
\putrule from -0.5 .15 to -0.5 -.15 
\putrule from 0.1 .15 to 0.1 -.15 
\putrule from 0.5 .15 to  0.5 -.15 
\putrule from -0.1 .15 to  -0.1 -.15 

\endpicture
\]
\caption{Illustration of Theorem \ref{theorem one well} with touchdown far away from the maxima of $f$.}
\label{fig th 2 b}
\end{figure}

In \cite{GG07},  for ``M''-shaped profiles in dimension one,
situations similar to Corollary \ref{cor 1} and Theorem~\ref{single point quenching} 
(cf.~figures~\ref{fig cor 1}--\ref{fig th 2})
were observed numerically, with respectively two and a single touchdown point.
In the case of Corollary \ref{cor 1} we here do not know whether there are two points or more. 
On the other hand, for some other ``M''-shaped profiles
(roughly, intermediate between figure \ref{fig cor 1} and \ref{fig th 1}), touchdown on a whole interval containing $0$ was observed numerically, which we are presently unable to confirm analytically.
These seem to be difficult questions. In this connection, we stress that results asserting the finiteness of the singular set for one-dimensional or radial problems (see \cite{CM89}),
based on reflection techniques, are essentially restricted to the case of constant or monotone coefficients.
Also, it was shown in \cite{Ve93} that for the nonlinear heat equation in $\R^n$ with constant coefficients, 
the blowup set has Hausdorff dimension at most $n-1$,
but the methods in \cite{Ve93} do not seem to apply to the present situation.

\medskip
\eject

Our last example shows that more complicated behaviors can occur. Namely the touchdown set 
 can be concentrated near two arbitrarily given points.
In the case when $\Omega$ is a ball, we can construct radially symmetric profiles for which 
 the touchdown set is concentrated near two arbitrarily given $(n-1)$-dimensional spheres.

\begin{theorem}\label{two comp. quench set}
Let $p>0$. Let $\Omega\subset \mathbb{R}^n$ a smooth bounded domain. 

(i) {\rm (Touchdown set concentrated near two arbitrary points.)}
For any $x_1,x_2\in\Omega$ and any $\rho>0$,
there exist positive profiles $f\in E$ such that
$$\mathcal{T}_f \subset B(x_1,\rho)\cup B(x_2,\rho), \qquad \mathcal{T}_f \cap B(x_1,\rho)\neq \emptyset,
\qquad \mathcal{T}_f \cap B(x_2,\rho)\neq \emptyset.$$

(ii) {\rm (Touchdown set concentrated near two arbitrary spheres.)}
Let $\Omega = B_R\subset \mathbb{R}^n$, $0<r_1<r_2<R$, $\rho>0$ and set 
$A_i=\{x\in \mathbb{R}^n;\; |x|\in(r_i-\rho,r_i+\rho)\}$.
There exist positive, radially symmetric profiles $f\in E$ such that 
$$\mathcal{T}_f \subset A_1\cup A_2, \qquad \mathcal{T}_f \cap A_1\neq \emptyset,
\qquad \mathcal{T}_f \cap A_2\neq \emptyset.$$
\end{theorem}

\begin{figure}[h]
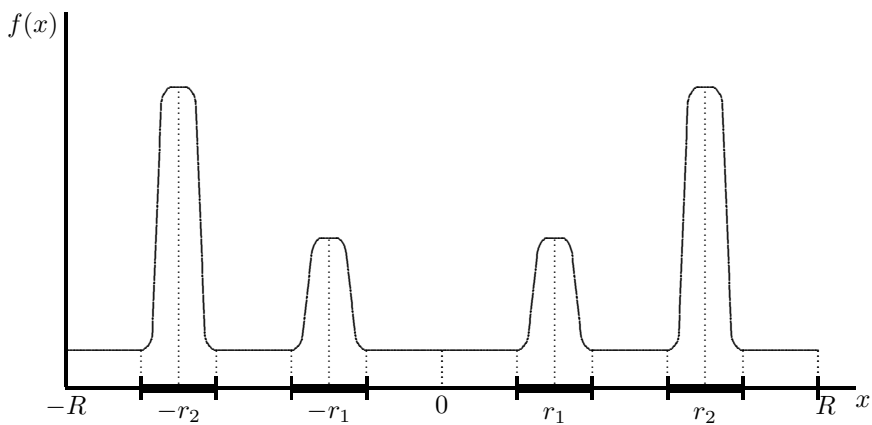

\[
\beginpicture
\setcoordinatesystem units <1cm,1cm>
\setplotarea x from -6 to 6, y from -1 to 5

\setdots <0pt>
\linethickness=1pt
\putrule from -5 0 to 5.5 0
\putrule from -5 0 to -5 5

\setlinear
\plot -1.27 1.8  -1.15 0.7 / 
\plot -1 0.5  1 0.5 / 
\plot -1.85 0.7  -1.73 1.8 / 
\plot -5 0.5  -4 0.5 /
\plot -3 0.5  -2 0.5 / 
\plot 1.27 1.8  1.15 0.7 / 
\plot 1 0.5  1 0.5 / 
\plot 1.85 0.7  1.73 1.8 / 
\plot 5 0.5  4 0.5 / 
\plot 3 0.5  2 0.5 /

\plot -3.27 3.8  -3.15 0.7 /
\plot -3.85 0.7  -3.73 3.8 /
\plot 3.27 3.8  3.15 0.7 /
\plot 3.85 0.7  3.73 3.8 /

\plot -1.4 2  -1.6 2 / 
\plot 1.4 2  1.6 2 / 
\plot -3.4 4  -3.6 4 / 
\plot 3.4 4  3.6 4 / 

\setquadratic
\plot -1.4 2  -1.32 1.94  -1.27 1.8 /
\plot -1.15 0.7  -1.1 0.57  -1 0.5 / 
\plot -2 0.5  -1.9 0.57  -1.85 0.7 / 
\plot -1.6 2  -1.68 1.94  -1.73 1.8 /
\plot 1.4 2  1.32 1.94  1.27 1.8 /
\plot 1.15 0.7  1.1 0.57  1 0.5 / 
\plot 2 0.5  1.9 0.57  1.85 0.7 / 
\plot 1.6 2  1.68 1.94  1.73 1.8 /

\plot -3.4 4  -3.32 3.94  -3.27 3.8 / 
\plot -3.15 0.7  -3.1 0.57  -3 0.5 / 
\plot -4 0.5  -3.9 0.57  -3.85 0.7 / 
\plot -3.6 4  -3.68 3.94  -3.73 3.8 / 
\plot 3.4 4  3.32 3.94  3.27 3.8 / 
\plot 3.15 0.7  3.1 0.57  3 0.5 / 
\plot 4 0.5  3.9 0.57  3.85 0.7 / 
\plot 3.6 4  3.68 3.94  3.73 3.8 / 

\setdots <2pt>
\setlinear
\plot -1.5 0  -1.5 2 /
\plot -1 0  -1 0.5 /
\plot -2 0  -2 0.5 /
\plot 0 0  0 0.5 /
\plot 1.5 0  1.5 2 /
\plot 1 0  1 0.5 /
\plot 2 0  2 0.5 /
\plot 5 0  5 0.5 /

\plot -3.5 0  -3.5 4 /
\plot -3 0  -3 0.5 /
\plot -4 0  -4 0.5 /
\plot 0 0  0 0.5 /
\plot 3.5 0  3.5 4 /
\plot 3 0  3 0.5 /
\plot 4 0  4 0.5 /
\plot 5 0  5 0.5 /

\put {$x$} [lt] at 5.5 -.1
\put {$R$} [ct] at 5.1 -.1
\put {$0$}   [ct] at 0 -.1
\put {$-r_1$}  [ct] at -1.5 -.19
\put {$-r_2$}  [ct] at -3.5 -.19
\put {$r_1$}  [ct] at 1.5 -.26
\put {$r_2$}  [ct] at 3.5 -.26
\put {$-R$}  [ct] at -5 -.1
\put {$f(x)$} [rt] at -5.1 5

\setdots <0pt>
\linethickness=3pt
\putrule from -2 0 to -1 0
\putrule from 2 0 to 1 0
\putrule from -4 0 to -3 0
\putrule from 4 0 to 3 0

\linethickness=1pt
\putrule from -2 .15 to -2 -.15 
\putrule from -1 .15 to  -1 -.15 
\putrule from 2 .15 to 2 -.15 
\putrule from 1 .15 to  1 -.15 

\putrule from -4 .15 to -4 -.15 
\putrule from -3 .15 to  -3 -.15 
\putrule from 4 .15 to 4 -.15 
\putrule from 3 .15 to  3 -.15 
\putrule from 5 .15 to  5 -.15 

\endpicture
\]
\caption{Illustration of Theorem \ref{two comp. quench set}(ii) for $n=1$.}
\label{fig th two comp quench}
\end{figure}

\begin{remark}
(i) In Theorem \ref{two comp. quench set}, the touchdown set in particular has at least two connected components if $\rho$ is sufficiently small (at least four if $n=1$).
The profile in Theorem \ref{two comp. quench set}(i) is obtained, by a limiting argument, by constructing a two-bump profile, where each bump is contained in $B(x_1,\rho)$, $B(x_2,\rho)$ respectively, and smoothly varying the height in each bump. 
For (ii), we follow the same idea but considering radially symmetric profiles and 
 replacing balls with annuli.

(ii) In the case of the one-dimensional nonlinear heat equation with constant coefficients, 
for any prescribed finite set, it was shown in \cite{Me92} that there exists an initial data 
for which the solution blows up exactly on this set. 
Such a construction does not seem easy to transpose to problem (\ref{quenching problem}).
\end{remark}

The proofs of the results in this subsection crucially depend on, rather delicate, stability properties of the touchdown set 
and time under small perturbations of the potential. We here state the following result,
 which may be of independent interest. Further results are given and proved in Section~\ref{stability results}.

\begin{theorem}[Continuity of the touchdown time and upper semi-continuity of the touchdown set]
\label{semi-continuity quenching set2}
Let $p>0$ and $\Omega\subset \mathbb{R}^n$ a smooth bounded domain.
Let $1\le q \le \infty$ with $q>\frac{n}{2}$, $B\subset \Omega$ a ball of radius $r>0$, $M\geq \mu>\mu_0(p,n) r^{-2}$
and set 
\begin{equation}\label{deftildeE}
\tilde E=\bigl\{f\in E;\, M\ge f\geq\mu\chi_B\bigr\}.
\end{equation}
For all $f\in \tilde E$ with $\mathcal{T}_f\subset\subset\Omega$
and all $\sigma>0$, there exists $\varepsilon>0$ such that,
$$\hbox{ if $g\in \tilde E$ and $\|g-f\|_q \leq \varepsilon$, \ then $|T_g-T_f|\le \sigma$ and 
$\mathcal{T}_g\subset \mathcal{T}_f+B(0,\sigma)$.}$$
\end{theorem}

\bigskip
On the other hand, we can show that the {\it continuity} of the touchdown set with respect to $f$ fails in general
-- see Proposition \ref{Tdiscontinuous}, which will be a consequence of the proof of Theorem \ref{two comp. quench set}.
Actually, considering the profile constructed in that proof and depicted in fig.~\ref{fig th two comp quench} for $n=1$,
it is shown that the touchdown points in the inner bumps
immediately disappear as soon as the height of this plateau is decreased.

\goodbreak

\begin{remark}
For results on continuity of the existence time in the case of blow-up problems, see 
\cite{BC87}, \cite{GRZ}, \cite{Qu03}, \cite{QS} and the references therein.
For results on the semi-continuity of the blow-up set, see~\cite{Me92}, \cite{AFPB}.
We note that the latter are restricted to one-dimensional problems,
due to the lack of estimates near every possible point in the blow-up set.
We are here able to avoid such restriction in the case of quenching problems, 
taking advantage of a time integrability property of the RHS of the PDE in (\ref{quenching problem})
up to the quenching time (see the first step of the proof of Theorem \ref{semi-continuity quenching set}).
However we have to face some additional difficulties due to the lack of a type I estimate up to the boundary.
\end{remark}

The outline of the rest of the paper is as follows.
In Section 2, we give some basic estimates for the touchdown time $T$, which will be useful in the sequel.
Sections 3 and 4 are devoted to the proofs of Theorems~\ref{local result dim n} and \ref{global result dim n},
based on refinements of the approach in \cite{GS}. Namely, in Section 3, we establish a type I estimate
for the touchdown rate away from the boundary. In Section 4, we establish a no-touchdown criterion
under an (optimal) smallness condition on $f$, assuming a local type I estimate.
We then combine it with the estimate obtained in Section 3 to conclude the proof.
In Section~5 we prove results on the continuity of the touchdown time and the semicontinuity of the touchdown set
under small perturbations of the permittivity profile $f$.
In Section 6, we then apply them, along with Theorem~\ref{local result dim n} and our type I estimate,
to establish Theorems \ref{single point quenching}, \ref{theorem one well} and \ref{two comp. quench set}.   

\goodbreak


\section{Basic estimates for the touchdown time} 

The following simple estimates will be useful in the sequel.

\begin{lemma}[Lower estimate for $T$]\label{lower estimate for T} 
Let $u$ be the solution of (\ref{quenching problem}). Then, $T\geq T_*:= \dfrac{1}{(p+1)\| f\|_{\infty}}$ and, for any $\tau\in (0,1)$ we have
$\| u(t_0)\|_{\infty} \leq 1 - \tau$,
where $ t_0 = t_0(\tau)=\dfrac{1- \tau^{p+1}}{(p+1) \| f\|_{\infty}}$. 
\end{lemma}

\begin{proof}
Let $y(t) \in C^1 (0,T_*)$ be the solution of the problem
\begin{equation*}\left\lbrace
\begin{array}{ll}
y' = \dfrac{\| f\|_{\infty}}{(1-y)^p}, & \mbox{for } t>0, \\
y(0) = 0. &  
\end{array}\right.
\end{equation*}
We have
$$\displaystyle\int_0^t y'(1-y)^p dt = \| f\|_{\infty} t$$
that is,
$$\dfrac{1}{p+1}  - \dfrac{(1-y(t))^{p+1}}{p+1} = \| f \|_{\infty} t, \quad\hbox{ for all $t\in [0,T_*)$}.$$
Since, by the comparison principle, $T\ge T_*$ and $u(t,x)\le y(t)$ 
for all $(t,x)\in [0,T)\times\Omega$,
it follows that $\|u(t_0(\tau))\|_{\infty} \leq y(t_0(\tau))=1-\tau$.
\end{proof}

\begin{lemma}[Upper estimate for $T$]\label{upper estimate for T}
Assume that $B(0,r)\subset\Omega$, $f\geq \mu\chi_{B(0,r)}$, with 
$\mu>\mu_0(p,n) r^{-2}$, where $\mu_0(p,n)$ is defined in \eqref{DefMu0}.
 Let $u$ be the solution of problem (\ref{quenching problem}).
Then $T<\infty$ and $T$ satisfies the upper bound 
$$T\leq \dfrac{1}{(p+1)(\mu - \mu_0(p,n)r^{-2})}.$$
\end{lemma}

\begin{proof}
Let $\tilde{\varphi}$ denote the first eigenfunction of $-\Delta$ in $H^1_0(B(0,r))$,
with $\|\tilde{\varphi} \|_{L^1}=1$ and set $\lambda_r$, the corresponding eigenvalue, i.e $\lambda_r = \lambda_1 r^{-2}$.
Set $y(t)=\int_{B(0,r)} u(t)\tilde{\varphi}$ and note that $y<1$ on $[0,T)$.
Multiplying (\ref{quenching problem}) by $\tilde{\varphi}$, integrating by parts over $B(0,r)$ and using Jensen's inequality
(in view of the convexity of the function $(1-s)^{-p}$), we obtain
$$y'(t) \ge \mu (1-y(t))^{-p}-\lambda_r y(t),\quad 0<t<T.$$
An elementary computation shows that 
$$\displaystyle\max_{0\le s<1}s(1-s)^p=\displaystyle\max_{0\le X<1}X^p-X^{p+1}=\dfrac{p^p}{(p+1)^{p+1}}.$$
It follows that
$$y'(t) \ge \Bigl(\mu -\dfrac{p^p}{(p+1)^{p+1}}\lambda_r \Bigr)(1-y(t))^{-p},\quad 0<t<T.$$
The conclusion follows by integration.
\end{proof}


\section{Qualitative Type I estimate}

Following the approach in \cite{GS}, a key ingredient in 
the proofs of Theorems \ref{local result dim n} and \ref{global result dim n} is the following type~I estimate for $u$ away from the boundary.

\begin{proposition}[type~I estimate]\label{type I qualitative}
Under assumption (\ref{condMB}), the solution $u$ of problem (\ref{quenching problem}) satisfies
\begin{equation}\label{EstimTypeI}
u(t,x) \le 1 - \gamma \delta(x) (T-t)^{\frac{1}{p+1}}, \quad \hbox{ for all $t\in [0,T)$ and $x\in \Omega$},
\end{equation}
where $\gamma$ depends only on $p, \Omega, M, r$.
\end{proposition}

A similar estimate is given in \cite[Theorem 1.2]{GS}, except that the constant $\gamma$ depends on $u$ in an unspecified way. 
We stress that the precise dependence of $\gamma$ is here a key feature, not only for the proofs
of Theorems \ref{local result dim n} and \ref{global result dim n},
but also in view of the stability results for the touchdown time and set in Section~5,
which require uniform type I estimates with respect to the permittivity profile $f$.
Proposition~\ref{type I qualitative} will be proved by means of the maximum principle 
applied to an auxiliary function of the form
\begin{equation}\label{function J}
J(t,x) = u_t - \eps a(x) h(u).
\end{equation}
We here follow the approach of \cite{GS}, which was a modification of the Friedman-McLeod method (\cite{FM}; see also \cite{G1}).
The main new ideas in \cite{GS} were to construct $h$ as a suitable perturbation of the nonlinearity 
and $a(x)$ as an appropriate function vanishing on the boundary.
In order to obtain the precise dependence of $\gamma$,
special care is here necessary in the construction and in the estimates of the function $a$.

\subsection{Basic computation for the function $J$.}

The basic computation for the function $J$ is contained in the following lemma.
Although it is close to \cite[Lemma 2.1]{GS}, we give the proof for convenience and completeness.

\begin{lemma}\label{basic computation}
Let  
$a\in C^2(\Omega)$ be a positive function.
Let $u$ be the solution of  (\ref{quenching problem}) 
and let $J$ be given by (\ref{function J}) in $(0,T)\times\Omega$, 
where 
\begin{equation}\label{function h}
h(u) = (1-u)^{-p} + 1,\quad\  0\le u<1.
\end{equation}
Then 
\begin{equation}\label{basic computation 1}
J_t - \Delta J - pf(x)(1-u)^{-p-1}J = \eps  \Theta\quad\hbox{ in $(0,T)\times \Omega$,} 
\end{equation}
where
\begin{equation}\label{basic computation 2}
\Theta = p a(x) f(x) (1-u)^{-p-1} + ah''(u) |\nabla u|^2 + 2h'(u)  \nabla a\cdot\nabla u + h(u) \Delta a.
\end{equation}
Moreover, we have $h''(u)>0$ for all $u\in [0,1)$ and
\begin{equation}\label{basic computation 3}
\Theta \geq \underbrace{ p a(x) f(x) (1-u)^{-p-1}}_{\tau_1} + \underbrace{h(u) \Delta a(x)}_{\tau_2}
 - \underbrace{\dfrac{h'^2(u)|\nabla a(x)|^2}{a(x)h''(u)}}_{\tau_3}.
\end{equation}
\end{lemma}

\begin{proof}
We compute
\begin{eqnarray*}
J_t &=& u_{tt} -\eps  a(x)h'(u)u_t, \\
\nabla J &=& \nabla u_t - \eps  \big( a(x)h'(u)\nabla u + h(u) \nabla a(x)\big), \\
\Delta J&=& \Delta u_t - \eps  \big( a(x)h'(u)\Delta u + a(x)h''(u) |\nabla u|^2 
  + 2h'(u)\nabla a(x)\cdot\nabla u + h(u) \Delta a(x) \big).
\end{eqnarray*}
Setting $g(u) = (1-u)^{-p}$ and omitting the variables $x,u$ without risk of confusion, we get
\begin{eqnarray*}
J_t - \Delta J &=& (u_t-\Delta u)_t - \eps  ah'(u_t- \Delta u) 
 +\eps  (ah''  |\nabla u|^2  |\nabla u|^2 + 2h'\nabla a\cdot\nabla u+ \Delta a) \\
               &=& fg'u_t - \eps  fah'g + \eps (ah''  |\nabla u|^2 + 2h'\nabla a\cdot\nabla u+ \Delta a).
\end{eqnarray*}
Using $u_t=J+\eps  ah$, we have
$$J_t- \Delta J - fg'J = \eps \Theta,$$
where
$$\Theta=fa(g'h-h'g)+ah''  |\nabla u|^2 + 2h'(u)\nabla a\cdot\nabla u+ \Delta a.$$
On the other hand, we have
\begin{equation}\label{hprimeu}
 h'(u)=p(1-u)^{-p-1},
\end{equation}
hence
\begin{eqnarray*}
g'h - h'g &=& p(1-u)^{-p-1} [(1-u)^{-p} + 1] - p (1-u)^{-2p-1} \\
          &=& p(1-u)^{-p-1},   
\end{eqnarray*}
which yields (\ref{basic computation 2}). Also, we have
\begin{equation}\label{hprimeprimeu}
h'' = [p(p+1) (1-u)^{-p-2}  > 0.
\end{equation}
Finally, since $a>0$, we may write
$$\Theta = p a(x) f(x) (1-u)^{-p-1} + h \Delta a + ah'' \left[ |\nabla u|^2 + 2\dfrac{h'(u)\nabla a\cdot\nabla u}{ah''}\right].$$
Since $|\nabla u|^2 + 2\dfrac{h'(u)\nabla a\cdot\nabla u}{ah''}\geq -\dfrac{h'^2 |\nabla a|^2}{a^2(h'')^2}$, inequality (\ref{basic computation 3}) follows.
\end{proof}

\subsection{Construction of the function $a(x)$.}

We shall apply Lemma \ref{basic computation}. 
In order to guarantee $\Theta\geq 0$, the negative term $\tau_3$ on the right-hand side of 
(\ref{basic computation 3}) must be absorbed by a positive contribution coming either from the term $\tau_1$, provided $f(x)>0$, or from the term $\tau_2$, provided $\Delta a(x)>0$. But $a(x)$ is positive and we require that it vanishes at the boundary,  
so we cannot have $\Delta a>0$ everywhere. 
Therefore, we shall consider a function $a(x)$ which is positive in $\Omega$ and suitably convex everywhere, 
except in a ball $B$ where $f$ is bounded away from zero.
A key point is here to obtain estimates of $a$ in terms of the radius of $B$, but independent of its location.

The following lemma gives the construction of the appropriate function $a(x)$.
In what follows we set
$$\Omega_r:=\{x\in \Omega;\ \delta(x)>r\},\qquad \omega_r:=\{x\in \Omega;\ \delta(x)<r\}.$$

\begin{lemma}\label{construction of a(x)}
Let 
\begin{equation}\label{function h 1}
h(u) = (1-u)^{-p} + 1.
\end{equation}
Let $r>0$, $y\in\Omega_{2r}$ and set $B=B_r(y)$. 
Then there exists a function $a\in C^2(\overline{\Omega})$ with the following properties:
\begin{equation}\label{construction of a 1 qual}
h h'' a\Delta a - h'^2|\nabla a|^2 \geq 0, \qquad \mbox{for all } x\in \overline{\Omega\setminus B}\mbox{ and all } 0\leq u<1,
\end{equation}
\begin{equation}\label{construction of a 2 qual}
C_1\delta^{p+1}(x)\leq a(x) \leq C_2\delta^{p+1}(x), \qquad \mbox{for all } x\in \overline{\Omega},
\end{equation}
\begin{equation}\label{construction of a 3 qual}
\|a\|_{C^2(\overline{\Omega})}\leq C_3,
\end{equation}
for some constants $C_1,C_2,C_3>0$ depending only on $p, \Omega, r$ (and not on $y$).
\end{lemma}

\begin{proof}
{\bf Step 1.} {\it Construction of $a(x)$ in $\Omega\setminus B$ and proof of (\ref{construction of a 1 qual}).}
We introduce a suitable harmonic function $\phi=\phi_y$, the unique smooth solution of the problem
\begin{equation}\label{elliptic problem}
\left.\begin{array}{ll}
\Delta \phi = 0, & x \in \Omega\setminus B, \\
\noalign{\vskip 1mm}
\phi = 0,        & x \in \partial \Omega, \\
\noalign{\vskip 1mm}
\phi =1,         & x \in \partial B.
\end{array}\right\rbrace
\end{equation}
The function $\phi$ is smooth, and by the strong maximum principle, we have $0<\phi <1$ in $\Omega\setminus B$.
Now, we set
\begin{equation}\label{defaphi}
a(x) = \phi^{p+1} (x), \qquad x\in \overline{\Omega\setminus B}
\end{equation}
and we compute 
\begin{equation*}
\nabla a = (p+1)\phi^p \nabla\phi, \qquad
\Delta a = (p+1)p\phi^{p-1} |\nabla \phi |^2 + \underbrace{(p+1)\phi^p\Delta\phi}_{=0}
\end{equation*}
in $\Omega\setminus B$.
It follows that
$$a\Delta a = (p+1)p \phi^{p+1}\phi^{p-1} |\nabla \phi |^2= \dfrac{p}{p+1} |\nabla a|^2, \qquad \mbox{in } \Omega \setminus B.$$
Since, on the other hand, we have
$$hh''  =  p(p+1)(1-u)^{-2p-2} + p(p+1) (1-u)^{-p-2} \geq \dfrac{p+1}{p} (h')^2,\qquad 0<u<1,$$
due to (\ref{function h 1}), property (\ref{construction of a 1 qual}) follows.

{\bf Step 2.} {\it Uniform estimates in $\Omega\setminus B$.}
We shall prove that 
\begin{equation}\label{construction of a 2b}
a(x) \ge C_1\delta^{p+1}(x) \quad \mbox{for all $x\in \overline\Omega\setminus B$},
\end{equation}
and
\begin{equation}\label{construction of a 3b}
\|a\|_{C^2(\overline{\Omega}\setminus B)}\leq C_3,
\end{equation}
for some constants $C_1, C_3>0$ depending only on $p, \Omega, r$.

For each $y\in \overline\Omega_{2r}$, the function $\phi_y(y+\cdot)$ is harmonic 
in $\{r<|z|< 2r\}$ with $\phi_y(y+\cdot)=1$ on $\{|z|=r\}$.
Consequently, by elliptic regularity, there exists a constant $C=C(n,r)>0$ such that
\begin{equation}\label{uniformHopf0}
\|\phi_y\|_{C^2(\{r\le |x-y|\le 3r/2\})}\le C,\quad\hbox{ for all $y\in \overline\Omega_{2r}$}.
\end{equation}
Since $\phi_y =1$ on $\partial B_r(y)$,
we deduce that there exists $\sigma=\sigma(n,r)\in (0,r/6 )$ such that
\begin{equation}\label{uniformHopf1}
\phi_y\ge 1/2 \quad\hbox{ in $\{r\le |x-y|\le r+3 \sigma\}$,\quad for all $y\in \overline\Omega_{2r}$}.
\end{equation}

Next we claim that there exists $c>0$ such that
\begin{equation}\label{uniformHopf2}
-\dfrac{\partial \phi_y}{\partial\nu}\ge c \quad\hbox{ on $\partial\Omega$,\quad for all $y\in \overline\Omega_{2r}$}.
\end{equation}
Assume for contradiction that there exist sequences $y_i\in \overline\Omega_{2r}$ and $x_i\in \partial\Omega$ such that
\begin{equation}\label{uniformHopf2b}
\dfrac{\partial \phi_{y_i}}{\partial\nu}(x_i)\to 0.
\end{equation}
We may assume $y_i\to y_0\in \overline\Omega_{2r}$ and $x_i\to x_0\in \partial\Omega$.
Set 
$$d=\delta(y_0)-r-2 \sigma\ge r-2 \sigma>0.$$
 For all large $i$, we have $\delta(y_i)>\delta(y_0)-\sigma$, 
hence $\delta(y_i)-r>d+\sigma$, so that $\phi_{y_i}$ is harmonic in $\omega_{d+\sigma}\subset\Omega\setminus B_r(y_i)$ 
with $\phi_{y_i}=0$ on $\partial\Omega$.
Applying elliptic regularity again, it follows that there exist $\alpha\in (0,1)$ and $C>0$ such that
\begin{equation}\label{uniformHopf3}
\|\phi_{y_i}\|_{C^{2+\alpha}(\overline \omega_{d})}\le C,\quad\hbox{ for all large $i$}.
\end{equation}
Up to extracting a subsequence, if follows that
\begin{equation}\label{uniformHopf4}
\phi_{y_i}\to \phi\quad\hbox{ in $C^2(\overline \omega_d)$,}
\end{equation}
where $\phi\ge 0$ is harmonic in $\omega_d$ and satisfies $\phi=0$ on $\partial\Omega$.
Moreover, by (\ref{uniformHopf2b}) we have $\frac{\partial \phi}{\partial\nu}(x_0)=0$.
By Hopf's Lemma, we deduce that 
\begin{equation}\label{uniformHopf5}
\phi\equiv 0 \quad\hbox{ in $\omega_d$.}
\end{equation}
Now, for large $i$, we have $\delta(y_i)<\delta(y_0)+\sigma$, hence 
$\delta(y_i)-r-3\sigma<d$, so that
$$\{r\le |x-y_i|\le r+ 3\sigma\}\cap\overline \omega_d\neq\emptyset.$$
But (\ref{uniformHopf4}) and (\ref{uniformHopf5}) then yield a contradiction with (\ref{uniformHopf1}).
The claim (\ref{uniformHopf2}) follows.

Now, arguing as for (\ref{uniformHopf3}), we have
\begin{equation}\label{uniformHopf6}
\|\phi_{y}\|_{C^2(\overline \omega_{r/2})}\le C(\Omega,r),\quad\hbox{ for all $y\in \overline\Omega_{2r}$}.
\end{equation}
Combining this with (\ref{uniformHopf2}), we deduce that there exists $\eta\in (0,r/2 )$ such that
\begin{equation}\label{uniformHopf6b}
\phi_{y}(x)\ge \frac{c}{2}\delta(x),\quad\hbox{ for all $x\in \omega_\eta$ and all $y\in \overline\Omega_{2r}$}.
\end{equation}
Since $\phi_y$ now satisfies
$$
\left.\begin{array}{ll}
\Delta \phi_y = 0, & x \in \Omega_\eta\setminus B_r(y) \\
\noalign{\vskip 1mm}
\phi_y \ge \frac{c\eta}{2},        & x \in \partial \Omega_\eta, \\
\noalign{\vskip 1mm}
\phi_y =1,         & x \in \partial B_r(y),
\end{array}\right\rbrace
$$
we deduce from the maximum principle that $\phi\ge \frac{c\eta}{2}$ in $\Omega_\eta\setminus B_r(y)$.
This along with (\ref{uniformHopf6b}) guarantees (\ref{construction of a 2b}).

Finally, for $x\in \Omega_\eta\setminus B_{3r/2}(y)$, we observe that $\phi_y$ is harmonic in $B_\eps(x)$ with $\eps=\min(\eta,r/2)$
and $0\le \phi_y\le 1$. It follows from elliptic regularity that 
that there exists a constant $C>0$ such that
\begin{equation}\label{uniformHopf7}
\|\phi_y\|_{C^2(B_{\eps/2}(x))}\le C,\quad\hbox{ for all $x\in \Omega_\eta\setminus B_{3r/2}(y)$ and all $y\in \overline\Omega_{2r}$}.
\end{equation}
Property (\ref{construction of a 3b}) is then a consequence of (\ref{uniformHopf0}), (\ref{uniformHopf6}) and (\ref{uniformHopf7}).

\goodbreak

{\bf Step 3.} {\it Extension to $B$.}
Since $a \in C^2(\overline B_{2r}(y)\setminus B_{r}(y))$ and $a$ satisfies (\ref{construction of a 3b}),
by standard properties of extension operators, the function $a$ can be extended in $\overline B_r(y)$ 
to a function $\tilde a$ such that 
\begin{equation}\label{extensionprop1}
\|\tilde a\|_{C^2(\overline B_{2r}(y))}\leq C_3.
\end{equation}
On the other hand, since $a=1$ on $\partial B_{r}(y)$, there exists $r_1\in (0,r)$ depending only on $C_3$ such that
\begin{equation}\label{extensionprop2}
\tilde a(x)\ge 1/2\quad\hbox{ for $r_1\le |x-y|\le r$.}
\end{equation}
 Fix a cutoff function $\psi\in C^2([0,\infty))$ such that $0\le \psi\le 1$,
$\psi(s)=0$ for $s\in [0,r_1]$ and $\psi(s)=1$ for $s\in [(r+r_1)/2,r]$, and define $a$ in $B_{r}(y)$ by
$$a(x):=1+(\tilde a(x)-1)\psi(|x-y|).$$ 
We thus obtain a function which satisfies $\|a\|_{C^2(\overline B_{r}(y))}\leq C_4(p,\Omega,r)$ 
and $a(x)\ge 1/2$ in $\overline B_r(y)$, 
owing to (\ref{extensionprop1}) and (\ref{extensionprop2}).
This, along with (\ref{construction of a 2b}) and (\ref{construction of a 3b}), guarantees (\ref{construction of a 3 qual})
 and the lower estimate in (\ref{construction of a 2 qual}).
Finally, the upper estimate in (\ref{construction of a 2 qual}) follows from (\ref{construction of a 3 qual}),
\eqref{defaphi}, (\ref{uniformHopf6}) and $\phi=0$ on $\partial\Omega$.
\end{proof}

\subsection{Proof of Proposition~\ref{type I qualitative}}

We shall also use the following lower bound for~$u_t$. 

\begin{lemma}\label{low bound u_t}
Under assumption (\ref{condMB}), for a given $t_0\in(0,T)$, the solution $u$ of problem (\ref{quenching problem}) satisfies
$$u_t (t,x) \geq c_0 e^{-c_1t}\delta(x), \qquad \mbox{for all } t\in \left[t_0,T\right)\mbox{ and } x\in\Omega,$$
with $c_0=c_0(\Omega,r,t_0)>0$ and $c_1=c_1(\Omega)>0$.
\end{lemma}

\begin{proof}
Let $x_0\in \Omega$ be such that $B=B(x_0,r)$. First, we observe that the function $v=u_t$
is a (classical) solution of the problem:
\begin{equation}\label{eqnut}
\left\lbrace\begin{array}{ll}
v_t - \Delta v = pf(x) (1-u)^{-p-1}v, & \mbox{ in } (0,T)\times \Omega, \\
v=0, & \mbox{ in } [0,T)\times \partial\Omega, \\
v(0,x)=f(x), &\mbox{ in } \Omega.
\end{array}\right.
\end{equation}
By the maximum principle, we thus have 
\begin{equation}\label{utsemigroup}
u_t\geq e^{t\Delta_\Omega}f\quad\hbox{ in $[0,T)\times \Omega$.}
\end{equation}
By (\ref{condMB}), we deduce that 
$$u_t\ge r\int_\Omega G_\Omega(t,x,y)\chi_B(y)\, dy.$$ 
Here, $e^{t\Delta_\Omega}$ and $G_\Omega$ are respectively the Dirichlet heat semigroup and heat kernel of $\Omega$.
It is known (see \cite{Dav} and also \cite{Zh}) that 
$$G_\Omega(t,x,y)\ge ce^{-c_1t}\delta(x)\delta(y),\quad t\ge t_0$$
with $c=c(t_0,\Omega)>0$ and $c_1=c_1(\Omega)>0$.
Consequently, since $\delta(x_0)\ge r$, we have
$$u_t\ge cre^{-c_1t}\delta(x)\int_{B(x_0,r/2)}\delta(y)\, dy\ge \frac{cr^2}{2}|B(0,r/2)|e^{-c_1t}\delta(x)$$
and the lemma follows.
\end{proof}

\begin{proof}[Proof of Proposition~\ref{type I qualitative}]
It is done in three steps.

\textit{Step 1: Preparations.} 
Let $J$ and $h$ be given by (\ref{function J}) and (\ref{function h 1}).
Owing to assumption (\ref{condMB}), upon replacing $r$ by $r/2$, we may assume that there exists 
$y\in \Omega$ such that $\delta(y)\ge 2r$ and 
\begin{equation}\label{f-lowerB}
f\ge r\quad\hbox{ on $B:=B_r(y)$.}
\end{equation}
Consider the function $a\in C^2(\overline\Omega)$ given by Lemma \ref{construction of a(x)}.
By (\ref{construction of a 2 qual}), we have
\begin{equation}\label{sigma}
\inf_{x\in B} a(x)\ge \sigma=\sigma(\Omega,p,r) := C_1r^{p+1}.
\end{equation}
Next, let $t_0 = \dfrac{1}{2(p+1)M}$, where $M$ is given by (\ref{condMB}). 
By Lemma \ref{lower estimate for T} we have $0<t_0<T$ and
\begin{equation}\label{boundut0}
 \|u(t,\cdot)\|_\infty \leq 1-2^{-1/(p+1)},\ \quad 0\le t\le t_0.
\end{equation}
We split the cylinder $\Sigma := \left( t_0,T\right) \times \Omega$ into three subregions as follows:
\begin{equation}\label{defSigma123}
\begin{array}{ll}
 \Sigma_1 &= (t_0,T) \times [\Omega \setminus B], \\
  \noalign{\vskip 1mm}
 \Sigma_2^\eta &= \left\lbrace (t,x)\in \left( t_0,T\right) \times B; \quad u(t,x)\geq 1-\eta \right\rbrace, \\
  \noalign{\vskip 1mm}
 \Sigma_3^\eta &= \left\lbrace (t,x)\in \left( t_0,T\right) \times B; \quad u(t,x)< 1-\eta \right\rbrace,
\end{array}
\end{equation}
where $\eta\in (0,1)$ will be specified later. 

\textit{Step 2: Parabolic inequality for $J$ in the regions $\Sigma_1$ and $\Sigma_2^\eta$.}
It follows from properties (\ref{basic computation 3}) in Lemma \ref{basic computation} and (\ref{construction of a 1 qual}) 
in Lemma \ref{construction of a(x)}, along with $a>0$, $f\geq 0$ in $\Omega$, and $h''>0$, that 
\begin{equation}\label{parab ineq J 1}
 J_t - \Delta J - pf(x)(1-u)^{-p-1} J \geq 0 \quad \mbox{ in } \Sigma_1.
\end{equation}
Next, in view of (\ref{function h 1}) and property (\ref{construction of a 3 qual}) in Lemma \ref{construction of a(x)}, we have 
$$|h\Delta a| \leq C_4 (1-u)^{-p}, \quad |h' \nabla a|\leq C_4 (1-u)^{-p-1}\quad \mbox{ in } \Sigma,$$
for some $C_4=C_4(\Omega,p,r)>0$. Also, from (\ref{function h 1}) and (\ref{sigma}) we get
$$ah'' \geq \sigma p(p+1) (1-u)^{-p-2}\quad \mbox{ in } (0,T) \times B.$$
Consequently, recalling the definition (\ref{basic computation 2}) of $\Theta$, it follows from 
(\ref{basic computation 3}, (\ref{f-lowerB}), (\ref{sigma}) that
\begin{eqnarray*}
 (1-u)^{p+1} \Theta &\geq & pf(x)a(x) +h\Delta a (1-u)^{p+1} - \dfrac{(h'|\nabla a|)^2}{ah''} (1-u)^{p+1} \\
 &\geq & pr\sigma - C_5(1-u) \geq pr\sigma - C_5 \eta \quad \mbox{ in } \Sigma_2^\eta,
\end{eqnarray*}
for some $C_5=C_5(\Omega,p,r)>0$. 
Choosing $\eta=\eta(\Omega,p,r)\in (0,1)$ small enough, 
we then deduce from~(\ref{basic computation 1}) that
\begin{equation}\label{parab ineq J 2}
 J_t - \Delta J - pf(x) (1-u)^{-p-1}J \geq 0 \quad \mbox{ in } \Sigma_2^\eta.
\end{equation}

\textit{Step 3: Control of $J$ on $\Sigma_3^\eta$ and conclusion.}
Now that $\eta$ has been fixed, using Lemma \ref{low bound u_t}
and (\ref{condMB}), (\ref{function h 1}),  (\ref{construction of a 2 qual}), (\ref{boundut0}),
we may choose $\eps=\eps(\Omega,p,r,M)>0$ small enough, such that
\begin{equation}\label{J ge 0 in S3}
 J\geq \delta(x) \left[c_0e^{-c_1M} - 2C_2\varepsilon \delta^p(x) (1-u)^{-p} \right] 
 \geq \delta(x) \left[c_0e^{-c_1M} - 2C_2 \varepsilon \delta^p(x) \eta^{-p}\right] \geq 0
 \quad \mbox{ in } \Sigma_3^\eta
\end{equation}
and
\begin{equation}\label{Jge 0 in t0}
\begin{array}{ll}
  J(t_0,x)&\geq \delta (x) \left[c_0e^{-c_1M} - 2C_2\varepsilon \delta^p(x) (1-\|u(t_0,\cdot)\|_\infty)^{-p} \right] \\
 \noalign{\vskip 1mm}
 &\geq \delta (x) \left[c_0e^{-c_1M} - 2^{1+\frac{p}{p+1}} C_2\varepsilon \delta^p(x) \right] \geq 0
\quad \mbox{ in } \overline{\Omega},
  \end{array}
\end{equation}
where $c_0,c_1$ are the constants in Lemma \ref{low bound u_t} and $C_2$ is the constant in (\ref{construction of a 2 qual}). 
Observe now that, as a consequence of (\ref{J ge 0 in S3}) and $\Sigma = \Sigma_1 \cup \Sigma_2^\eta \cup \Sigma_3^\eta$, we have
\begin{equation}\label{negative set J}
 \lbrace (t,x)\in\Sigma ; \ J(t,x) < 0\rbrace \subset \Sigma_1\cup \Sigma_2^\eta.
\end{equation}
Also, since $a=0$ on $\partial \Omega$, we have
\begin{equation}\label{boundary cond}
 J=0 \quad \mbox{on } (t_0,T) \times \partial \Omega.
\end{equation}
On the other hand, by standard parabolic regularity, we have
$$J\in C^{1,2}(\Sigma)\cap C([t_0,T)\times \overline{\Omega}).$$
It follows from (\ref{parab ineq J 1}), (\ref{parab ineq J 2}), (\ref{Jge 0 in t0})-(\ref{boundary cond}), and the maximum principle (see, e.g., \cite{QS}, Proposition 52.4 and Remark 52.11(a)) that
$$J\geq 0 \quad \mbox{ in } \Sigma.$$

Then, for $t_0<t<s<T$ and $x\in \Omega$, we have
$$u_t \geq \eps a(x) h(u) \geq \eps a(x)(1-u)^{-p}$$
and an integration in time gives
$$(1-u(t,x))^{p+1} \geq (p+1)\displaystyle\int_t^s u_t(1-u)^p \geq \eps a(x) (s-t).$$
Letting $s \to T$, we get
\begin{equation}\label{time integration}
(1-u(t,x))^{p+1} \geq (p+1) \eps a(x) (T-t) \quad \mbox{ in } \Sigma.
\end{equation}
In view of (\ref{construction of a 3 qual}), this implies (\ref{EstimTypeI}) in $[t_0,T)\times\Omega$
with $\gamma=\gamma(\Omega,p,r,M)>0$.
Due to (\ref{boundut0}), (\ref{condMB}), the estimate (\ref{EstimTypeI}) is true in $[0,t_0)\times\Omega$ as well,
for a possibly smaller constant $\gamma=\gamma(\Omega,p,r,M)>0$. 
\end{proof}


\section{Proof of Theorems \ref{local result dim n} and \ref{global result dim n}}

\subsection{No touchdown criterion under a local type I estimate.}

The following lemma enables one to exclude touchdown at a given interior point 
or on a neighborhood of $\partial\Omega$, under a suitable type~I estimate
and a smallness assumption on $f$.

\begin{lemma}\label{basic supersol}
Let $u$ be the solution of problem (\ref{quenching problem}). 
Let either 
$$D=B(x_0,b)\subset\subset\Omega \quad\hbox{ and }\quad \Gamma=\partial D,
\leqno(i)$$
or
$$D=\Omega\setminus\omega\ \hbox{ for some $\omega\subset\subset\Omega$}\quad\hbox{ and }\quad\Gamma=\partial\omega,
\leqno(ii)$$
or
$$\Omega=(-R,R),\quad D=(a,R)\ \hbox{ for some $a\in(-R,R)$} \quad\hbox{ and }\quad \Gamma=\{a\}.
\leqno(iii)$$
Assume 
\begin{equation}\label{hyp basic supersol}
u\le 1-k(T-t)^{\frac{1}{p+1}} \quad\hbox{ on $[0,T)\times \Gamma$}
\end{equation}
for some $k>0$. If
\begin{equation}\label{hypf basic supersol}
\|f\|_{L^\infty(D)}<\frac{k^{p+1}}{p+1},
\end{equation}
then $\mathcal{T} \cap  D=\emptyset$. In addition, in case (ii) we have $\mathcal{T}\cap \partial\Omega = \emptyset$, and in case (iii), $R\notin \mathcal{T}$.
\end{lemma}

\begin{remark}
Condition (\ref{hypf basic supersol}) is essentially optimal. Indeed, considering (\ref{quenching problem}) with $f(x)\equiv 1$
and leaving the boundary conditions apart, 
we see that the ODE solution $y(t)=1-[1-(p+1)t]^{1/(p+1)}$ satisfies (\ref{hyp basic supersol}) with $k=(p+1)^{\frac{1}{p+1}}$ and $T=1/(p+1)$,
so that one could not take a larger value of the constant in the RHS of (\ref{hypf basic supersol}).
\end{remark}

\begin{proof}
We use a simplification of a comparison argument from \cite{GS}
(where the comparison was done with a selfsimilar supersolution, instead of a separated variable supersolution).
We define the comparison function
$$w(t,x) := y(t) \psi(x)  \quad \text{ for } (t,x)\in 
[0,T)\times \overline{D},$$
where $y(t)$ is defined by
$$
y(t) =1-k(T-t)^{\frac{1}{p+1}}.
$$
Here, in case (i), $\psi$ is given by
$$\psi(x) := 1- \sigma \left( 1-\dfrac{|x-x_0|^2}{b^2}\right)$$
for $\sigma\in (0,1)$ to be chosen below and, in case (ii), $\psi$ is the solution of the problem
$$
\left\lbrace \begin{array}{rrl}
\Delta \psi=0, &&x\in D, \\
\psi = 1, &&x\in \partial \omega, \\
\psi = 1-\sigma, &&x\in \partial \Omega.
\end{array} \right.
$$ 
Observe that
$$1-\sigma < \psi(x)<1,\quad x\in D,$$
by the strong maximum principle.
In case (iii), similarly to (ii), we set $\psi(x)=1 - \sigma (x-a)/(R-a)$ for $x\in [a,R]$.
In particular, owing to \eqref{hypf basic supersol}, we note that in all cases,
\begin{equation}\label{signw}
w\ge 0  \quad \text{ in }  [0,T)\times \overline{D}.
\end{equation}

We compute, in $(0,T)\times D$:
\begin{eqnarray*}
 w_t - \Delta w - f(x) (1-u)^{-p} & =& y'(t) \psi (x) - y(t)\Delta \psi (x) - f(x) (1-y(t)\psi(t))^{-p} \\
                 &\ge & \frac{k}{p+1} (T-t)^{-\frac{p}{p+1}} \psi(x) - y(t)\Delta \psi(x) - f(x) (1-y(t))^{-p} \\
                 & = &  \left( \frac{k}{p+1}\psi(x) - f(x)k^{-p} \right) (T-t)^{-\frac{p}{p+1}} - y(t)\Delta \psi(x).
\end{eqnarray*}
Moreover, we have $\Delta \psi = \frac{2\sigma}{b^2}$ in $B(x_0, b)$ in case (i),
and $\Delta \psi =0$ in $D$ in cases (ii) and (iii).
In all cases,
using assumption \eqref{hypf basic supersol} and taking $\sigma>0$ small enough,
 it follows that
\begin{equation}\label{eqnwcomp}
w_t -\Delta w -f(x)(1-w)^{-p} \ge \left( \frac{k}{p+1}(1-\sigma) - f(x)k^{-p}\right) T^{-\frac{p}{p+1}} -  \frac{2\sigma}{b^2} \ge 0
\quad\hbox{in $[0,T)\times D$.}
\end{equation}

We next look at the comparison on the parabolic boundary of $[0,T)\times D$.
On the one hand, by~\eqref{signw}, we have
\begin{equation}\label{boundcondw1}
w(0,x) \ge 0= u(0,x) \quad\hbox{ in $\overline D$.}
\end{equation}
On the other hand, using $\psi= 1$ on $\Gamma$ and (\ref{hyp basic supersol}), we have
\begin{equation}\label{boundcondw2}
\begin{array}{lll}
w(t,x) 
= 1 - k (T-t)^{\frac{1}{p+1}} \geq u(t,x)
\quad\hbox{ in $[0,T)\times \Gamma$.}
\end{array}
\end{equation}
Moreover, in case (ii) (resp., (iii)), we have, by \eqref{signw},
\begin{equation}\label{boundcondw3}
w(t,x) \ge 0 = u(t,x) \quad\hbox{in $[0,T)\times \partial\Omega$ (resp., $[0,T)\times \{-R\}$).}
\end{equation}

By (\ref{eqnwcomp}), (\ref{boundcondw1}), (\ref{boundcondw2}) and (\ref{boundcondw3}) (in cases (ii) and (iii)),
along with the comparison principle and $y(t)\leq 1$ for all $t\in[0,T)$, we conclude that
\begin{equation}\label{usmallD}
u(t,x) \leq w(t,x)\leq \psi(x) \quad\hbox{ in $(0,T)\times D$.}
\end{equation}

 In all cases, since $\psi$ is uniformly smaller than $1$ in compact subsets of $D$,
it follows from \eqref{usmallD} that $\mathcal{T} \cap D=\emptyset$. We also see that in case (ii),  
$\psi$ is uniformly smaller than $1$ in a neighborhood of $\partial\Omega$, so we can rule out quenching at the boundary.
For the case (iii), the conclusion follows similarly.
\end{proof}


\subsection{Proof of Theorem \ref{local result dim n}.}
We shall apply case (i) of Lemma~\ref{basic supersol}. 
Let $\gamma$ be given by estimate (\ref{EstimTypeI}), and assume 
$$f(x_0)<\dfrac{ (\gamma \delta(x_0))^{p+1}}{p+1}.$$
Pick $k\in (0,\gamma \delta(x_0))$ such that 
$$f(x_0)<\dfrac{k^{p+1}}{p+1}<\dfrac{ (\gamma \delta(x_0))^{p+1}}{p+1}.$$
By estimate \eqref{EstimTypeI}, together with the continuity of $f$,
conditions \eqref{hyp basic supersol} and \eqref{hypf basic supersol} are satisfied in $D=B(x_0,b)$ for $b>0$ sufficiently small.
We can then conclude from Lemma~\ref{basic supersol} that $x_0\not\in \mathcal{T}$,
which proves Theorem \ref{local result dim n} with $\gamma_0=\frac{\gamma^{p+1}}{p+1}$.
\hfill $\square$

\subsection{Proof of Theorem \ref{global result dim n}.}
Let $\gamma$ be given by estimate (\ref{EstimTypeI}), and assume (\ref{hyp global result dim n}) with 
$\gamma_0:=\frac{\gamma^{p+1}}{p+1}$.
Applying case (ii) of Lemma~\ref{basic supersol} with $k=\gamma\, {\rm dist}(\omega,\partial\Omega)$,
it follows that $\mathcal{T}\subset\overline\omega$. 
Finally, we note that for any $x\in \partial \omega$, our assumption implies
$f(x)<\gamma_0 \delta^{p+1}(x)$, so that $x\not\in\mathcal{T}$ by Theorem \ref{local result dim n}.
Therefore $\mathcal{T}\subset\omega$ and the theorem is proved.
\hfill $\square$

\subsection{Proof of Corollary \ref{cor 1}.}
Assertion (i) follows from Theorem~\ref{local result dim n}.
Assertion (ii) follows by applying Theorems~\ref{local result dim n} and \ref{global result dim n},
and then Lemma~\ref{upper estimate for T}.
\hfill $\square$


\section{Stability results for the touchdown time and touchdown set} \label{stability results}

One of the main ingredients in the proofs of 
Theorems \ref{single point quenching}, \ref{theorem one well} and \ref{two comp. quench set}
 is the stability of the touchdown time and touchdown set under small perturbations of the potential $f$.

Recalling the definition in \eqref{hypf}, we denote by $U:E\ni f\mapsto U_f$ the semiflow generated by problem~(\ref{quenching problem}).
Namely, $u=U_f(t,\cdot)$ is the maximal classical solution of (\ref{quenching problem}).
We recall that its existence time and touchdown set are respectively denoted by $T_f\in (0,\infty]$ 
and $\mathcal{T}_f\subset\overline\Omega$. We start with a more or less standard continuous dependence
property of the solution itself with respect to $f$.

\begin{proposition}[Continuity of $U$ from $L^q$ to $L^\infty$]\label{Lq continuity}
Let $1\le q \le \infty$ with $q>\frac{n}{2}$. Let $f\in E$ and let $0<t_0<T_f$.
For all $\sigma>0$, there exists $\varepsilon>0$ such that 
$$\hbox{ if $g\in E$ and $\|g-f\|_q \leq \varepsilon$, then $T_g>t_0$ and 
$\displaystyle\sup_{t\in [0,t_0]} \|U_g-U_f\|_\infty\le \sigma$.}$$
\end{proposition}

For the stability of the touchdown time and set, the local type I estimate (\ref{EstimTypeI}) in Proposition~\ref{type I qualitative}
 plays a crucial role. A uniform version is actually needed.
To this end, for given $\gamma>0$, we set 
$$E_\gamma=\Bigl\{g\in E;\ T_g<\infty \ \hbox{ and }\ 
U_g(t,x) \le 1- \gamma \delta(x) (T_g -t)^{\frac{1}{p+1}} \ \hbox{ for all } (t,x)\in [0,T_g) \times \Omega
\Bigr\}.$$

\begin{proposition}[Continuity of the touchdown time]\label{quenching time continuity}
Let $1\le q \le \infty$ with $q>\frac{n}{2}$, $\gamma>0$ and let $f\in E$ be such that $T_f<\infty$ and $\mathcal{T}_f\cap\Omega\neq \emptyset$.
For all $\sigma>0$, there exists $\varepsilon>0$ such that 
$$\hbox{ if $g\in E_\gamma$ and $\|g-f\|_q \leq \varepsilon$, then $|T_g - T_f| < \sigma$.}$$
\end{proposition}

\begin{theorem}[Upper semi-continuity of the touchdown set]\label{semi-continuity quenching set}
Let $1\le q \le \infty$ with $q>\frac{n}{2}$, $\gamma, M>0$ and let $f\in E$ be such that $T_f<\infty$ and $\mathcal{T}_f\subset\subset\Omega$.
For all $\sigma>0$, 
there exist $\varepsilon,\kappa>0$ such that, if
$$\hbox{ $g\in E_\gamma$, $\|g\|_\infty\le M \text{ and } \|g-f\|_q \leq \varepsilon$,}$$
then 
$$U_g(t,x)\le 1-\kappa \qquad \mbox{in } [0,T_g)\times \bigl(\overline\Omega\setminus (\mathcal{T}_f+B(0,\sigma))\bigr),$$
hence in particular
$$\mathcal{T}_g\subset \mathcal{T}_f+B(0,\sigma).$$
\end{theorem}

\begin{remark}\label{uniform type 1}
The assumption $g\in E_\gamma$,
i.e. estimate (\ref{EstimTypeI}) with a uniform constant,
can be guaranteed by assuming
$\mu\chi_B\leq g\leq M$, where $M,r>0,\ B\subset \Omega$ is a ball of radius $r$
and $\mu>\mu_0(p,n) r^{-2}$ (cf.~Lemma~\ref{upper estimate for T}).
This is a consequence of Proposition \ref{type I qualitative}.
\end{remark}

\begin{remark}
(i) Theorem \ref{semi-continuity quenching set}
in particular proves that $\mathcal{T}_g$ is also a compact subset of $\Omega$, provided $g$ is close enough to $f$ in $L^q$ norm.
It is unknown whether the compactness assumption on $\mathcal{T}_f$ can be removed. This would be true if we 
knew the analogue of estimate (\ref{EstimTypeI}) without the factor distance to the boundary.

(ii) To ensure that the touchdown set $\mathcal{T}_f$ is compact, we can consider $f$ small enough near the boundary (apply Theorem \ref{global result dim n}), or $\Omega$ convex and $f$ non-increasing near the boundary in the outer direction (this is proved in \cite{G08} by a moving planes argument). Also, if we consider $0<p<1$, then the touchdown set is compact for any $f$ (see \cite{GS}). 
\end{remark}

We note that Theorem \ref{semi-continuity quenching set2} is a direct consequence of 
 Proposition \ref{quenching time continuity} and Theorem \ref{semi-continuity quenching set}, 
 together with Proposition \ref{type I qualitative}.

The semi-continuity property of $\mathcal{T}_f$ in Theorem \ref{semi-continuity quenching set2} can be expressed as 
$$d\bigl(\mathcal{T}_g,\mathcal{T}_f\bigr)\to 0,
\quad\hbox{ as $g\to f$ in $L^q$, $g\in \tilde E$},$$
where $\tilde E$ is defined in \eqref{deftildeE} and
\begin{equation}\label{defhausdorff}
d(A,B)=\displaystyle\sup_{x\in A} d(x,B)
\end{equation}
denotes the usual Hausdorff semi-distance.
Our next result shows that the {\it continuity} of the touchdown set with respect to $f$ fails in general.

\begin{proposition}[Non continuity of the touchdown set]\label{Tdiscontinuous} 
Let $p>0$ and $\Omega=B_R\subset \mathbb{R}^n$.
Let $1\le q <\infty$ with $q>\frac{n}{2}$.
One can find $B\subset \Omega$ a ball of radius $r>0$, $M\ge\mu>\mu_0(p,n) r^{-2}$,
a function $f\in \tilde E$ with $\mathcal{T}_f\subset\subset\Omega$ 
and a sequence $g_i\in \tilde E$, such that
$$g_i\to f \ \hbox{in $L^q$}\quad\hbox{and}\quad 
\liminf_{i\to\infty}\  d\bigl(\mathcal{T}_f,\mathcal{T}_{g_i}\bigr)>0.$$
\end{proposition}

Proposition \ref{Tdiscontinuous} will be proved in the next section,
along with Theorem \ref{two comp. quench set}.

\begin{proof}[Proof of Proposition \ref{Lq continuity}]
Set $h(z)=(1-z)^{-p}$. Note that, for any $0<M<1$,
we have $0<h'(z) <L(M):=p(1-M)^{-p-1}$ for $0<z<M$.
Now, fix
$$M:=\max_{0\leq t\leq t_0} \|U_f(t)\|_\infty<1$$
and define
$$\tau_g := \sup \left\lbrace t\in[0,T_g);\ \|U_g(s)\|_\infty \leq \dfrac{1+M}{2} \ \hbox{ for all } s\in [0,t]\right\rbrace.$$

By the variation-of-constants formula and the $L^p$-$L^q$-estimates for the linear heat semigroup, we have
\begin{eqnarray*}
\| (U_f-U_g)(t)\|_\infty &\leq & \displaystyle\int_0^t \|e^{(t-s)\Delta} (f(x)h(U_f)-g(x)h(U_g))\|_\infty \,ds \\
&\leq & \displaystyle\int_0^t (4\pi(t-s))^{-\frac{n}{2q}} \|f(x)h(U_f)-g(x)h(U_g)\|_q \,ds.
\end{eqnarray*}
Now, for all $0<t\leq \min\lbrace t_0 ,\tau_g\rbrace$, we obtain
\begin{eqnarray*}
 \|f(\cdot)h(U_f(t,\cdot))-g(\cdot)h(U_g(t,\cdot))\|_q
&=& \|f(\cdot)h(U_f(t,\cdot))-f(\cdot)h(U_g(t,\cdot))+f(\cdot)h(U_g(t,\cdot))-g(\cdot)h(U_g(t,\cdot))\|_q \\
&\leq & \|f\|_q \| h(U_f(t))-h(U_g(t))\|_\infty + \|h(U_g(t))\|_\infty \|f-g \|_q \\
&\leq & \|f\|_q L\left(\textstyle\frac{1+M}{2}\right) \|(U_f-U_g)(t)\|_\infty + \left(\textstyle\frac{1-M}{2}\right)^{-p}\|f-g\|_q.
\end{eqnarray*}
Therefore,
\begin{eqnarray*}
\|(U_f-U_g)(t)\|_\infty &\leq & (4\pi)^{-\frac{n}{2q}} \|f\|_q L
\left(\textstyle\frac{1+M}{2}\right) \displaystyle\int_0^t(t-s)^{-\frac{n}{2q}}\|(U_f-U_g)(s)\|_\infty \,ds \\
& & + (4\pi)^{-\frac{n}{2q}} \| f-g\|_q\left(\textstyle\frac{1-M}{2}\right)^{-p} \displaystyle\int_0^t(t-s)^{-\frac{n}{2q}}\,ds  \\
&\leq & C_1  \displaystyle\int_0^t(t-s)^{-\frac{n}{2q}}\|(U_f-U_g)(s)\|_\infty \,ds +C_2\| f-g\|_q,
\end{eqnarray*} 
for all $0<t\leq \min \lbrace t_0, \tau_g\rbrace$, where $C_1,C_2>0$ are two constants independent of $g$. Now, applying Gronwall's Lemma, we obtain
\begin{equation}\label{Gronwall}
\|(U_f-U_g)(t)\|_\infty \leq C_3 \| f-g \|_q.
\end{equation}
If we consider $\varepsilon>0$ small enough, then, for all $g$ such that $\| f-g \|_q<\varepsilon$, we have
$$\|(U_f-U_g)(t)\|_\infty \leq \dfrac{1-M}{4}, \qquad \mbox{for } 0<t\leq t_1:=\min\lbrace t_0, \tau_g\rbrace.$$
Therefore, $\|U_g(t_1)\|_\infty < \dfrac{1+M}{2}$, hence $\tau_g > t_1$, i.e. $\tau_g> t_0$.
We deduce that $T_g>t_0$ and the result then follows from (\ref{Gronwall}).
\end{proof}

\begin{proof}[Proof of Proposition \ref{quenching time continuity}]
The lower semicontinuity of the touchdown time is a consequence of Proposition \ref{Lq continuity}.
Therefore we may always assume $T_g>T_f$. By assumption, there exist sequences $x_i\in \Omega$ and $t_i<T_f$
such that $x_i\to x_0\in \Omega$, $t_i\to T_f$ and $U_f(t_i,x_i)\to 1$. We may assume $\delta(x_i)\ge c>0$.

Now fix $0<\lambda<1$, pick $i$ such that $U_f(t_i,x_i)\ge \lambda$ and next take $0<\alpha <1$. 
As a consequence of Proposition \ref{Lq continuity}, there exists $\varepsilon>0$ such that if $\|g-f\|_q<\varepsilon$, then $U_g(t_i,x_i)\geq \alpha \lambda$. Since $g\in E_\gamma$, we have
$$1-\alpha \lambda \geq 1-U_g(t_i,x_i) \geq c \gamma (T_g - t_i)^{\frac{1}{p+1}} \geq c \gamma (T_g - T_f)^{\frac{1}{p+1}}.$$
Therefore,
$$\limsup_{\|g-f\|_q\to 0} T_g\leq T_f+\left(\dfrac{1-\alpha \lambda}{c\gamma}\right)^{p+1}.$$
The result follows by letting $\alpha\to 1$ and then $\lambda\to 1$.
\end{proof}

\begin{proof}[Proof of  Theorem \ref{semi-continuity quenching set}]
The proof is more delicate. It is based on parabolic regularity, comparison arguments
and uniform H\"older estimates in time for $u$ up to the touchdown time. 
The latter follow from a key integrability property in time for the RHS of the PDE in (\ref{quenching problem}) (see \eqref{Keyint}),
which is a consequence of the type I estimate.

{\bf Step 1.} {\it Uniform H\"older estimates in time.}
For each $\eta>0$, we recall the notation
$$\Omega_\eta:= \left\lbrace x\in \Omega; \delta(x)>\eta \right\rbrace.$$
We claim that for all $\eta>0$ and $\beta\in (0,1/(p+1))$, there exists $C(\eta,\beta)>0$ such that 
for all $g\in E_\gamma$ with $\|g\|_\infty\le M$, we have
\begin{equation}\label{control v(T_eps)}
U_g(T_g,x)\le U_g(t,x)+ C(\eta,\beta)(T_g-t)^\beta, \quad 0<t<T_g, \ x\in \Omega_\eta.
\end{equation}

Set $a=\dfrac{p}{p+1} <1$ and take $\eta>0$.
Let $g\in E_\gamma$ with $\|g\|_\infty\le M$. We have 
\begin{equation}\label{Keyint}
|\partial_tU_g - \Delta U_g| = g(x)(1-U_g(t,x))^{-p} \leq C(\eta)(T_g - t)^{-a} \quad\hbox{ in $(0,T_g)\times \Omega_{\eta/2}$},
\end{equation}
with $C(\eta)>0$ independent of $g$. Moreover, by Lemma \ref{lower estimate for T},
 we have $T_g\ge \tau_0$ with $\tau_0\in (0,1)$ independent of $g$.

Now, for $\tau\in (0,\tau_0/2)$, we define $w(t,x)=\tau^{a}U_g(t,x)$, which satisfies
$$|w_t-\Delta w| \leq C(\eta) \quad \mbox{and} \quad 0\leq w\leq 1
\quad \hbox{ in $(0,T_g-\tau)\times \Omega_{\eta/2}$},$$
along with $w(0,\cdot)\equiv 0$.
By interior parabolic regularity (see e.g. \cite[Theorem 48.1]{QS}), 
for all $r\in(1,\infty)$, we deduce
$$\|w_t\|_{L^r(\Sigma_{g,\tau})} + \|D^2w\|_{L^r(\Sigma_{g,\tau})} \leq C(\eta,r),
\quad \hbox{ where } \Sigma_{g,\tau}:=(0,T_g-\tau)\times \Omega_\eta,$$
with $C(\eta,r)>0$ independent of $g$ and $\tau$.
Using Sobolev embedding in the time variable, we obtain, for all $\alpha\in (0,1)$,
$$\|w\|_{C^\alpha_t(\Sigma_{g,\tau})}\leq C(\eta,r).$$
Consequently, we have 
\begin{equation}\label{Holder estimate}
U_g(s,x)-U_g(t,x) \leq C(s-t)^\alpha (T_g - s)^{-a}, \quad 0<t<s<T_g, \ x\in \Omega_\eta.
\end{equation}
Here and until the end of Step 1, $C>0$ denotes a positive constant, depending on $\eta$, $\alpha$,
but independent of $g$.

Now, for fixed $0<t<T_g$, we consider the sequence $s_i = T_g -(T_g- t)2^{-i}$, which satisfies
\begin{equation}\label{sequence}
s_{i+1} - s_i = (T_g - t)2^{-i-1} = T_g - s_{i+1}.
\end{equation}
Fix $\alpha\in (a,1)$. From (\ref{Holder estimate}) and (\ref{sequence}), we have
$$
U_g(s_{i+1},x)-U_g(s_i,x) \leq  C(s_{i+1} - s_i)^\alpha (T_g - s_{i+1})^{-a}
= C (T_g -s_{i+1})^{\alpha -a} 
= C [(T_g- t)2^{-i-1}]^{\alpha - a} 
$$
and iterating, we obtain
$$U_g(s_{i+1},x)-U_g(t,x) \leq C(T_g -t)^{\alpha - a}\sum_{j=0}^i 2^{-(j+1)(\alpha-a)}.$$
Claim (\ref{control v(T_eps)}) follows by letting $i\to \infty$.

{\bf Step 2.} {\it No touchdown away from $\mathcal{T}_f$ and from $\partial\Omega$.}
Let $\sigma,\eta>0$. We claim that there exists $\kappa=\kappa(\sigma)>0$
(independent of $\eta$) and $\eps=\eps(\sigma,\eta)>0$, 
such that for all $g\in E_\gamma$ with $\|g\|_\infty\le M$, if $\|g-f\|_q\le \eps$, then
\begin{equation}\label{U_gbound1}
U_g(T_g,x)\le 1-2\kappa\qquad \mbox{in } \overline{\Omega_\eta} \setminus A_\sigma,
\end{equation}
where $A_\sigma:=\mathcal{T}_f+B(0,\sigma)$.

Choose any $\beta\in (0,1/(p+1))$. As a consequence the definition of $\mathcal{T}_f$, 
there exists $\kappa=\kappa(\sigma)\in (0,1/5)$ such that
\begin{equation}\label{u(T_0)}
U_f(t,x) \leq 1-5\kappa \qquad \mbox{in } [0,T_f)\times (\Omega \setminus A_\sigma).
\end{equation}
Set $t_0=\max\bigl(0,T_f-\bigl(\frac{\kappa}{C_\eta}\bigr)^{1/\beta}\bigr)$,
where $C_\eta=C(\eta,\beta)$ is given by (\ref{control v(T_eps)}).
By Proposition \ref{Lq continuity}, there exists $\varepsilon=\varepsilon(\sigma,\eta)$ such that, for all $g\in E$, if $\|g-f\|_q\le\varepsilon$, then 
$$U_g(t_0,x)\le U_f(t_0,x)+\kappa\le 1-4\kappa
\qquad \mbox{in } [0,t_0]\times (\Omega \setminus A_\sigma).$$
Applying (\ref{control v(T_eps)}), if follows that 
\begin{eqnarray*}
U_g(T_g,x)
&\le& U_g(t_0,x)+C_\eta(T_g-T_f)^\beta+C_\eta(T_f-t_0)^\beta \\
&\le& 1-3\kappa+C_\eta(T_g-T_f)^\beta
\qquad \mbox{in } \overline{\Omega_\eta} \setminus A_\sigma.
\end{eqnarray*}
The Claim then follows from Proposition \ref{quenching time continuity}.

{\bf Step 3.} {\it No touchdown near $\partial\Omega$ and conclusion.}
We shall use a supersolution argument to exclude touchdown on $\Omega\setminus\Omega_\eta$ 
for $g$ close to $f$ and some $\eta>0$ (independent of $g$).

Since $\mathcal{T}_f\subset\subset \Omega$, we may fix $\sigma_0>0$ sufficiently small, such that $A_{\sigma_0}\subset\subset \Omega$.
Set $\kappa=\kappa(\sigma_0)$, given by Step 2.
Let $W(x) = 1-2\kappa + K\phi(x)$, where $K=M \kappa^{-p}$ and $\phi$ is the solution of $-\Delta \phi =1$ in $\Omega$, with $\phi=0$ on $\partial\Omega$.
We can choose $\eta>0$ small enough such that  $K\phi (x)<\kappa$ in $\Omega\setminus \Omega_\eta$, so that
$$-\Delta W = K \geq \dfrac{M}{(1-W)^p} \quad \mbox{in } \Omega\setminus\Omega_\eta.$$
Taking $\eta>0$ smaller if necessary, we may also assume $\partial\Omega_\eta\cap A_{\sigma_0}=\emptyset$.
For all $g\in E_\gamma$ with $\|g\|_\infty\le M$ and $\|g-f\|_q\le \eps(\sigma_0,\eta)$, it then follows from \eqref{U_gbound1}
and $\partial_tU_g\ge 0$ that
$$U_g(t,x)\le 1-2\kappa\le W(x)  \qquad \mbox{on } [0,T_g)\times \partial\Omega_\eta.$$
Since $W\ge 0$, it follows from the comparison principle, applied on $[0,T_g)\times (\Omega\setminus \Omega_\eta)$, that
\begin{equation}\label{U_gbound2}
U_g(t,x)\le W(x) \le 1-\kappa \qquad \mbox{in } [0,T_g)\times (\overline\Omega\setminus \Omega_\eta).
\end{equation}
Finally combining \eqref{U_gbound2} and \eqref{U_gbound1} with the $\eta$ just chosen and any $\sigma>0$,
we conclude that, for all $g\in E_\gamma$ with $\|g\|_\infty\le M$, if $\|g-f\|_q\le \min(\eps(\sigma,\eta),\eps(\sigma_0,\eta))$, 
then
$$U_g(t,x)\le W(x) \le 1-\kappa \qquad \mbox{in } [0,T_g)\times (\overline\Omega\setminus A_\sigma).$$
The Theorem follows.
\end{proof}

\goodbreak


\section{Proof of Theorems \ref{single point quenching}, \ref{theorem one well} and \ref{two comp. quench set}}

\begin{proof}[Proof of Theorem \ref{single point quenching}]
{\bf Step 1.} {\it Estimates of $U_g$.}
Let $g$ satisfy the assumptions of the Theorem for some $\eps>0$.
Since $U_g$ (and $U_f$) is radially symmetric, we shall indifferently write $U_g(t,x)$ or $U_g(t,r)$ with $r=|x|$.
It is known from  \cite{DL89}, \cite{G1}, \cite{G08} that $\mathcal{T}_f=\{0\}$.

We first observe that, taking $\eps>0$ small enough, assumptions \eqref{gprimerho} and \eqref{gLq} guarantee that
\begin{equation}\label{gbigrho}
g(x)>\mu_0 \rho^{-2}\quad\hbox{on $B_\rho$.}
\end{equation}
Indeed for given $\delta>0$, we may choose $s_0\in (0,(R-\rho)/2]$ sufficiently small 
(depending only on $f,\delta$), such that 
$$\sup\,\bigl\{|f(r+s)-f(r)|;\, 0\le r\le \rho,\, 0\le s\le 2s_0\bigr\}\le \delta/2.$$
By \eqref{gprimerho}, for all $r\in [0, \rho]$ and $s\in [0,2s_0]$, we then have
\begin{eqnarray*}
g(r)
&\ge& g(r+s)-\eps R \\
&\ge&f(r)-|f(r+s)-f(r)|-|g(r+s)-f(r+s)|-\eps R \\
&\ge& f(r)-|g(r+s)-f(r+s)|-\eps R-\delta/2.
\end{eqnarray*}
Averaging in $s\in [s_0,2s_0]$ and using H\"older's inequality and \eqref{gLq}, we obtain, for all $r\in [0, \rho]$,
\begin{eqnarray*}
g(r)
&\ge& f(r)-s_0^{-1}\int_{s_0}^{2s_0}|g(r+s)-f(r+s)|\, ds-\eps R-\delta/2 \\
&\ge& f(r)-s_0^{-n}\int_{s_0}^{2s_0}|g(r+s)-f(r+s)|(r+s)^{n-1}\, ds-\eps R-\delta/2 \\
&\ge& f(r)-C(n,R,q)s_0^{-n}\|g-f\|_q-\eps R-\delta/2\ge f(r)-\delta,
\end{eqnarray*}
for $\eps>0$ sufficiently small. In view of our assumptions on $f$, property \eqref{gbigrho} follows by choosing $\delta$ 
sufficiently small.

Owing to \eqref{gbigrho}, we have $T_g<\infty$ and, by Remark \ref{uniform type 1}, 
$g\in E_\gamma$ for some $\gamma>0$ independent of $g$.
Also, in view of Theorem \ref{semi-continuity quenching set}, we may assume $T_g>t_0:=T_f/2$ and
\begin{equation}\label{U_gbound3}
U_g(t,x)\le 1-\kappa \qquad \mbox{in } [0,T_g)\times \{\rho/4\le |x|\le R\},
\end{equation}
for some $\kappa>0$ independent of $g$. By parabolic estimates, it follows that
\begin{equation}\label{U_gbound4}
\|U_g\|_{C^{1+\nu/2,2+\nu}([0,T_g]\times \{\rho/2\le |x|\le R\})} \le C_1.
\end{equation}
for some $C_1, \nu>0$ independent of $g$.

Next we claim that, for any given $t_1\in (0,T_f)$, we have
\begin{equation}\label{U_gClaim1}
\hbox{$T_g>t_1$ and $U_g(t_1,\cdot)$ converges to $U_f(t_1,\cdot)$ in $C^2(\overline B(0,\rho /2))$ as $\eps\to 0$.}
\end{equation}
The fact that $T_g>t_1$ for $\eps>0$ small follows from Proposition \ref{quenching time continuity}.
To prove the convergence, we first note that, by Proposition \ref{Lq continuity}, we have 
\begin{equation}\label{U_gCV0}
\sup_{t\in [0,t_1]} \|U_g(t,\cdot)-U_f(t,\cdot)\|_\infty\to 0,\quad \hbox{ as $\eps\to 0$.}
\end{equation}
In particular we can assume $\sup_{t\in [0,t_1]} \|U_g(t,\cdot)\|_\infty\le c<1$.
Moreover, our assumptions guarantee that $\|g\|_{C^1(\overline B(0,\rho))}\le C$, with $C>0$ independent of $g$.
It then follows from standard parabolic estimates that 
$$\|U_g\|_{C^{\nu/2,\nu}([0,t_1]\times \overline B(0,3\rho /4))} \le C$$
for some $\nu>0$, and next that 
$$\|U_g\|_{C^{1+\nu/2,2+\nu}([0,t_1]\times \overline B(0,\rho /2))} \le C.$$
Using compact embeddings, we deduce that, for any sequence $g_i$ 
satisfying the assumptions of the Theorem with $\eps=\eps_i\to 0$,
there exists a subsequence of $U_{g_i}(t_1,\cdot)$ which converges in $C^2(\overline B(0,\rho/2))$ to some limit $W$.
By \eqref{U_gCV0} we must have $W=U_f(t_1,\cdot)$ and property \eqref{U_gClaim1} follows.

{\bf Step 2.} {\it Monotonicity properties of $U_g$.}
We first claim that 
\begin{equation}\label{U_gClaim2}
\partial_r U_f<-c_0r \quad \text{in} \ [t_0,T_f)\times (0,\rho],
\end{equation} 
for some constant $c_0>0$.
To prove \eqref{U_gClaim2}, we set $B_+(0,R)=B(0,R)\cap\{x_1>0\}$. 
It follows from our assumptions that $z:=\partial_{x_1} U_f\le 0$ in $(0,T_f)\times B_+(0,R)$.
Recalling that $f\in C^1(\overline B_\rho)$ and using parabolic regularity, we deduce that $z$ is a (strong) subsolution
 of the heat equation 
in $Q:=(0,T_f)\times B_+(0,\rho)$, namely:
$$z_t-\Delta z = \dfrac{pf(0)z}{(1-U_f)^{p+1}}\le 0 \quad \text{a.e. in } Q.$$
Since $z=0$ on $\{x_1=0\}$ it follows from the strong maximum principle and the Hopf Lemma that
$$z(t,x_1,0,\cdots,0)\le -c_0x_1 \quad \text{for all } (t,x_1)\in [t_0,T_f)\times [0,\rho].$$
Claim \eqref{U_gClaim2} follows by observing that $\partial_r U_f(t,x)=\partial_{x_1}U_f(t,|r|,0,\cdots,0)$.

We next choose $t_1\in (t_0,T_f)$ such that 
\begin{equation}\label{U_gbound5}
C_1|T_f-t_1|^{\nu/2}\le \frac{c_0\rho}{8},
\end{equation} 
where the constants $C_1, c_0$ are given by \eqref{U_gbound4}, \eqref{U_gClaim2}, respectively.
We claim that if $\eps$ is sufficiently small, then
\begin{equation}\label{U_gClaim3}
\partial_r U_g(t,\rho/2)\le -\frac{c_0\rho}{16}, \quad\hbox{ for all $t\in [t_1,T_g)$.}
\end{equation} 
To prove \eqref{U_gClaim3}, we observe that, by \eqref{U_gClaim2} and \eqref{U_gClaim1}, 
if $\eps$ is sufficiently small, then we have 
$$\partial_rU_g(t_1,\rho/2)\le \partial_rU_f(t_1,\rho/2)+\frac{c_0\rho}{4}\le -\frac{c_0\rho}{4}.$$
Applying \eqref{U_gbound4}, \eqref{U_gbound5} and Proposition \ref{quenching time continuity},
 we deduce that, if $\eps$ is sufficiently small then, for all $t\in [t_1,T_g)$,
\begin{eqnarray*}
\partial_rU_g(t,\rho/2)
&\le& \partial_rU_g(t_1,\rho/2)+C_1|T_f-t_1|^{\nu/2}+C_1|T_g-T_f|^{\nu/2} \\
&\le& -\frac{c_0\rho}{4}+\frac{c_0\rho}{8}+C_1|T_g-T_f|^{\nu/2} \\
&\leq & -\frac{c_0\rho}{4}+\frac{c_0\rho}{8}+\frac{c_0 \rho}{16} = -\frac{c_0 \rho}{16},
\end{eqnarray*}
which proves \eqref{U_gClaim3}.

Then we claim that there exists a constant $c_1>0$ such that,
if $\eps$ is sufficiently small,
\begin{equation}\label{U_gClaim4}
\partial_r U_g(t_1,r)\le -c_1 r, \quad\hbox{ for all $r\in [0,\rho /2]$.}
\end{equation} 
To prove \eqref{U_gClaim4}, we note that, by \eqref{U_gClaim2}, 
there exists $\ell\in (0,\rho)$ such that $\partial^2_r U_f(t_1,r)\le -c_0/2$ for all $r\in [0,\ell]$
and $\partial_r U_f(t_1,r)\le -c_0\ell$ for all $r\in [\ell,\rho/2]$.
By property \eqref{U_gClaim1}, we deduce that if $\eps$ is sufficiently small, then
$\partial^2_rU_g(t_1,r)\le -c_0/4$ for all $r\in [0,\ell]$, which, after integration, gives
$$
\partial_r U_g(t_1,r) \leq - \frac{c_0}{4}r, \quad\hbox{ for all $r\in [0,\ell].$}
$$
Also by property \eqref{U_gClaim1}, if $\eps$ is small enough, we have
$$
\partial_r U_g (t_1,r) \leq -\frac{c_0\ell}{2} \leq -\frac{c_0\ell}{\rho} r, \quad\hbox{ for all $r\in [\ell,\rho/2].$}
$$
Hence, \eqref{U_gClaim4} follows by taking $c_1 = \min \left\lbrace \frac{c_0}{4} ,\frac{c_0 \ell}{\rho} \right\rbrace$.

{\bf Step 3.} {\it Auxiliary function and conclusion.}
In what follows, omitting the subscript $g$ without risk of confusion, we will use the notation $u=U_g$. 
Following the method in \cite{FM}, we define the auxiliary function
$$J(t,r) := w(t,r) + \eta a(r) h(u),\quad (t,r)\in [t_1,T_g)\times [0,\rho/2],$$
where $w(t,r) = r^{n-1} u_r, \ \ a(r) = r^n, \ \ h(u) = (1-u)^{-\gamma}$ and $\eta,\gamma>0$ are constants to be chosen later.

We first look at the parabolic boundary of $[t_1,T_g)\times (0,\rho/2)$. By \eqref{U_gClaim4}, for all $r\in [0,\rho/2)$, we have
\begin{equation}\label{t_1 single point}
J(t_1,r)  =  r^{n-1} u_r(t_1,r) + \eta \frac{r^n}{(1-u (t_1,r))^\gamma} \\
          \leq  r^n \left(-c_1 + \frac{\eta}{(1-\|u(t_1)\|_\infty)^\gamma}\right) \leq 0,        
\end{equation}
provided $\eta \leq c_1  (1-\|u(t_1)\|_\infty)^\gamma$.
We also have $J(t,0) = 0$, for all $t\in [t_1,T_g)$, and 
by \eqref{U_gClaim3} and \eqref{U_gbound3}, we have, for all $t\in [t_1,T_g)$,
\begin{equation}\label{rho/2 single point}
J\big(t,\rho/2\big)  =  \left( \frac{\rho}{2}\right)^{n-1} u_r (t,\rho/2) + \eta \frac{(\rho/2)^n}{(1-u (t,\rho/2))^\gamma} 
          \leq  \left( \frac{\rho}{2}\right)^n \left(- \frac{c_0 }{8} + \frac{\eta}{\kappa^\gamma}\right) \leq 0,        
\end{equation}
provided $\eta \leq \dfrac{c_0 \kappa^\gamma}{8}$.

Next, we note that $u$ and $w$ respectively solve the equations
\begin{equation*}\begin{array}{l}
u_t - u_{rr} - \dfrac{n-1}{r} u_r = g(r)H(u), \\
w_t - w_{rr} + \dfrac{n-1}{r} w_r = g(r) H'(u) w + g'(r) r^{n-1} H(u),
\end{array}
\end{equation*}
where $H(u)=(1-u)^{-p}$.
Omitting the variables $r$ and $u$ from now on without risk of confusion, we compute 
\begin{eqnarray*}
J_t &=& w_t + \eta a h' u_t, \\
J_r &=& w_r + \eta a' h + \eta a h' u_r, \\
J_{rr} &=& w_{rr} + \eta a'' h + 2\eta a'h' u_r + \eta a h'' u_r^2 + \eta a h' u_{rr}. 
\end{eqnarray*}
Using these identities and $w = J - \eta a h$, we obtain
(a.e. in $[t_1,T_g)\times (0,\rho/2)$):
\begin{equation*}\begin{array}{l}
   J_t - J_{rr} + \frac{n-1}{r} J_r \\
   \noalign{\vskip 2mm}
 \qquad \qquad = w_t - w_{rr} + \frac{n-1}{r} w_r + \eta a h' \left( u_t- u_{rr} - \frac{n-1}{r} u_r \right) \\
 \noalign{\vskip 2mm}
 \qquad \qquad \ \ + \frac{n-1}{r} \eta a' h - \eta a'' h - 2 \eta a'h'u_r - \eta a h'' u_r^2 + 2 \eta a h' \frac{n-1}{r} u_r \\
  \noalign{\vskip 2mm}
 \qquad \qquad = gH'w + g'r^{n-1}H + \eta a h' gH - 2 \eta h'w - \eta a h'' u_r^2  \\
 \noalign{\vskip 2mm}
 \qquad \qquad \leq [p g (1-u)^{-p-1} - 2\gamma \eta  (1-u)^{-1-\gamma}]w 
 + \gamma \eta g r^n (1-u)^{-p-1-\gamma} + g' r^{n-1}(1-u)^{-p} \\
 \noalign{\vskip 2mm}
 \qquad \qquad = b(t,r) J + r^{n-1} (1-u)^{-p} [g' - (p-\gamma)\eta g r (1-u)^{-\gamma-1} 
 + 2 \gamma \eta^2 r(1-u)^{p-2\gamma -1} ] \\
 \noalign{\vskip 2mm}
 \qquad \qquad = b(t,r) J + r^{n-1} (1-u)^{-p} \big[  g' - \eta r (1-u)^{-\gamma-1}\big( (p-\gamma) g - 2\gamma \eta (1-u)^{p-\gamma}  \big)  \big],
\end{array} \end{equation*}
where $b$ is a bounded function on $(0,\rho/2)\times [t_1, T_g-\tau]$ for each $\tau>0$.
For any $\gamma \in [0,p)$, using assumption \eqref{gprimerho}, it follows that
$$ J_t - J_{rr} + \frac{n-1}{r} J_r - b(t,r) J
  \le r^n (1-u)^{-p} \big[  \eps - \eta (1-u)^{-\gamma-1}\big( (p-\gamma)g - 2\gamma \eta \big)  \big].$$
Recalling from Step 1 that $g$ is uniformly positive in $B_\rho$, we can choose $\eta$ small enough such that 
$$(p-\gamma) g(r) \ge 3\gamma \eta \quad \text{on } B_{\rho/2}.$$
Taking $\eps\le \gamma \eta^2$, we finally obtain
$$  J_t - J_{rr} + \frac{n-1}{r} J_r - b(t,r) J
 \le r^n (1-u)^{-p} \big(\eps -\gamma \eta^2\big)\le 0 \quad \text{a.e. in $[t_1,T_g)\times (0,\rho/2)$}.$$

In view of this inequality, together with \eqref{t_1 single point} and \eqref{rho/2 single point}, it follows from the maximum principle that $J\leq 0$, hence
$$u_r \leq - \eta r (1-u)^{-\gamma}, 
\quad\hbox{ for all $(t,r) \in [t_1,T_g)\times [0,\rho/2).$}$$
After integrating this inequality in space, we obtain
$$(1-u(t,r))^{\gamma +1}\ge \frac{\gamma +1}{2}\eta r^2,
 \quad\hbox{ for all $(t,r) \in [t_1,T_g)\times [0,\rho/2),$}$$
 which guarantees $\mathcal{T}_g = \{0\}$.
\end{proof}

\begin{remark}
(i) In view of the previous proof, we obtain the following estimate of the final touchdown profile of the solution near the origin:
$$1-u(T,r) \geq c r^{\frac{2}{p+1}+\eps}, \quad\text{as} \ r\to 0, $$ 
for any $\eps>0$ and $c=c(\eps)>0$.
For more accurate results regarding the quenching profile 
in the case $f=Const.$, see \cite{FG93}.

(ii) We point out that under the stronger, global assumption
 that $g\in C^1(\overline\Omega)$ and $-M\le g'(r) \leq \eps r$ on $[0,R]$,
  we can prove single point touchdown at the origin without using the semicontinuity property of the touchdown set. 
To this end, in the above proof, one considers $J$ on the whole cylinder $[t_1,T_g)\times (0,R)$ and uses the Hopf lemma to ensure $J\le 0$ at $r=R$. In this case, we only use Proposition \ref{Lq continuity}, along with the hypothesis $\|g-f\|_q \leq \varepsilon$,
to ensure that $J(t_1,\cdot)\le 0$ at some time $t_1 <T_g$.
\end{remark}

\begin{proof}[Proof of Theorem \ref{theorem one well}]
It is a direct consequence of Theorem \ref{semi-continuity quenching set}, 
Remark \ref{uniform type 1} and Theorem~\ref{local result dim n}.
\end{proof}

\begin{proof}[Proof of Theorem \ref{two comp. quench set} and Proposition \ref{Tdiscontinuous}]
We just need to prove assertion (i) of Theorem \ref{two comp. quench set}. Assertion (ii) follows by the same arguments together with the radial symmetry of the domain and profile.

In order to construct this example, 
 we set $B_i=B(x_i,2r)$, with $r\in (0,\rho/2)$ chosen sufficiently small so that
$$\min \Big(\text{dist}(B_1,B_2),\text{dist}(B_1,\partial \Omega), \text{dist}(B_2,\partial\Omega)\Big)>r.$$
Denote $\tilde{B}_i=B(x_i,r)$.
Choose $\mu>\mu_0(p,n) r^{-2}$. By Lemma~\ref{upper estimate for T}, 
\begin{equation}\label{condg1}
\hbox{for any $g\in E$ such that $g\ge \mu\chi_{\tilde{B}_i}$ with $i=1$ or $2$,
we have $T_g\le  \frac{1}{(p+1)(\mu - \mu_0(p,n)r^{-2})}<\infty$.}
\end{equation}
Set 
$$\hat E=\bigl\{g\in E;\,2\mu\ge g\ge \mu\chi_{\tilde{B}_i}, \hbox{ with $i=1$ or $2$}\bigr\}.$$
By Theorems \ref{local result dim n} and \ref{global result dim n}, 
we deduce the existence of $\eta=\eta(p,\Omega,\mu,r)\in (0,\mu/2)$ such that 
\begin{equation}\label{condg}
\hbox{if $g\in \hat E$ and $g\le 2\eta$ in $\overline B_i$ for some $i\in\{1,2\}$,
then $\mathcal{T}_g\cap B_i=\emptyset$}
\end{equation}
and
\begin{equation}\label{condg2b}
\hbox{if $g\in \hat E$ and $g\le 2\eta$ in $\overline{\Omega}\setminus [B_1\cup B_2]$, 
then $\mathcal{T}_g\subset B_1\cup B_2.$}
\end{equation}

Now, for a fixed $q\ge 1$ and $\frac{n}{2}<q<\infty$, we consider a continuous map $h\mapsto f_h$ from $[\eta,2\mu]$ to $(E, \|\cdot\|_q)$ with the following properties.
Each $f_h$ 
satisfies
\begin{equation*}
f_h (x) = \left\{
\begin{array}{ll}
h & \text{for } x\in \tilde{B}_1, \\
2\mu + \eta - h & \text{for } x\in \tilde{B}_2, \\
\eta & \text{for } x\in \overline{\Omega}\setminus [B_1\cup B_2],
\end{array}
\right.
\end{equation*}
together with
$$\max_{\overline B_i} f_h = f_h(x_i), \qquad \text{for } i=1,2.$$
By \eqref{condg1}, \eqref{condg2b}, we have
\begin{equation}\label{quench in B1B2}
T_{f_h}<\infty \quad \text{and} \quad \mathcal{T}_{f_h} \subset B_1\cup B_2,
\qquad \text{for all } h\in [\eta,2\mu].
\end{equation}
In addition, by \eqref{condg}, \eqref{condg2b}, we have
\begin{equation}\label{quench in B2}
\mathcal{T}_{f_h} \subset B_2, \qquad \text{for all } h\in [\eta, 2\eta]
\end{equation}
and
\begin{equation}\label{quench in B1}
\mathcal{T}_{f_h} \subset B_1, \qquad \text{for all } h\in [2\mu-\eta , 2\mu]. 
\end{equation}

Now, we define
\begin{equation}\label{defhstar}
h^* := \inf\, \bigl\{ h\in [\eta,2\mu], \ \ \text{such that} \ \mathcal{T}_{f_h} \cap B_1 \neq \emptyset \bigr\}.
\end{equation}
By \eqref{quench in B2}, \eqref{quench in B1}, we know that $2\eta\le h^*\le 2\mu - \eta$. 
By the definition of $h^*$ and \eqref{quench in B1B2}, we have 
$$
\mathcal{T}_{f_h} \subset B_2, \qquad \text{for all} \ h\in [\eta , h^*),
$$
and there exists a sequence $h_i\downarrow h^*$ such that 
$$
\mathcal{T}_{f_{h_i}} \cap B_1 \neq \emptyset.
$$
Since $h\mapsto f_h$ is continuous in $L^q$ with $q$ as above, we may apply Theorem \ref{semi-continuity quenching set2}
 and \eqref{quench in B1B2} to deduce that
\begin{equation}\label{hstar2}
\mathcal{T}_{f_{h^*}} \cap B_1 \neq \emptyset \quad \text{and} \quad \mathcal{T}_{f_{h^*}} \cap B_2 \neq \emptyset.
\end{equation}
This completes the proof of Theorem \ref{two comp. quench set}.
Finally, in view of \eqref{defhstar} and the second part of \eqref{hstar2}, Proposition \ref{Tdiscontinuous} follows by considering the sequence $g_i=f_{h^*-1/i}$.
\end{proof}

\medskip
{\bf Acknowledgments.}
The second author is partially supported by the Labex MME-DII (ANR11-LBX-0023-01).

\end{document}